\documentclass[3p, reqno]{amsart}

\usepackage[utf8]{inputenc} 
\usepackage[english]{babel} 
\usepackage{amsmath}
\usepackage{amsthm}
\usepackage{amsfonts}
\usepackage{amssymb}
\usepackage{graphicx} 
\usepackage{float}	
\usepackage{geometry} 
\usepackage{fancyhdr} 
\usepackage{longtable}
\usepackage{listings}
\usepackage{subfigure}
\usepackage[colorlinks=true,linkcolor=blue,urlcolor=blue]{hyperref} 
\usepackage[usenames,dvipsnames]{color} 
\usepackage{multirow}
\usepackage{fancybox}
\usepackage{cases}
\usepackage{epsfig}
\usepackage{color}
\usepackage{pifont}
\usepackage{setspace}
\usepackage{multicol}
\usepackage{caption}

\newtheorem{thm}{Theorem}[section]
\newtheorem{prop}[thm]{Proposition}

\makeatletter
\renewcommand\theequation{\thesection.\arabic{equation}}
\@addtoreset{equation}{section}
\makeatother

\makeatletter
\DeclareRobustCommand*\cal{\@fontswitch\relax\mathcal}
\makeatother

\newcommand{\R}{
	\ensuremath{\mathbb R}}

\allowdisplaybreaks

\begin{document}

\title[A 2D model for hydrodynamics and biology coupling.]{A 2D model for hydrodynamics and biology coupling applied to \\ algae growth simulations.}

\author[O. Bernard]{Olivier Bernard}
\address{BioCore, INRIA, BP93, F-06902 Sophia-Antipolis Cedex}
\email{Olivier.Bernard@inria.fr}

\author[A.-C. Boulanger]{Anne-Céline Boulanger}
\address{UPMC Univ Paris 06, CNRS UMR 7598, Laboratoire Jacques-Louis Lions, F-75005, Paris;INRIA Rocquencourt, B.P.~105,F-78153 Le Chesnay Cedex}
\email{Anne-Celine.Boulanger@inria.fr}

\author[M.-O. Bristeau]{Marie-Odile Bristeau}
\address{INRIA Rocquencourt, B.P.~105,F-78153 Le Chesnay Cedex}
\email{Marie-Odile.Bristeau@inria.fr}

\author[J. Sainte-Marie]{Jacques Sainte-Marie}
\address{CETMEF, 2 boulevard Gambetta, F-60200 Compi\`egne}
\email{Jacques.Sainte-Marie@inria.fr}


\begin{abstract}

Cultivating oleaginous microalgae in specific culturing devices such as raceways is seen as a future way to produce biofuel. The complexity of this process coupling non linear biological activity to hydrodynamics makes the optimization problem very delicate. The large amount of parameters to be taken into account paves the way for a useful mathematical modeling.
Due to the heterogeneity of raceways along the depth dimension regarding temperature, light intensity or nutrients availability, we adopt a multilayer approach for hydrodynamics and biology.
For free surface hydrodynamics, we use a multilayer Saint-Venant model that allows mass exchanges, forced by a simplified representation of the paddlewheel. Then, starting from an improved Droop model that includes light effect on algae growth, we derive a similar multilayer system for the biological part. A kinetic interpretation of the whole system results in an efficient numerical scheme.
We show through numerical simulations in two dimensions that our approach is capable of discriminating between situations of mixed water or calm and heterogeneous pond. Moreover, we exhibit that \textit{a posteriori} treatment of our velocity fields can provide lagrangian trajectories  which are of great interest to assess the actual light pattern perceived by the algal cells and therefore understand its impact on the photosynthesis process.


\end{abstract}

\maketitle

\section{Introduction}
\label{introduction}

Recently, biofuel production from microalgae has proved to have a high potential for biofuel production \cite{Chisti2007,Wijffels2010}. Several studies have demonstrated that some microalgae species could store more than 50\% of their dry weight in lipids under certain conditions of nitrogen deprivency \cite{Chisti2007,RZBPBBT09,WilliamsLaurens2010} leading to productivities in a range of order larger than terrestrial plants. In this article, we focus on microalgae cultivation in raceways (also called high rate ponds), whose hydrodynamics has been less studied than the photobioreactor culturing devices \cite{Perner2002b,Luo2004,PruPotLeg06,Perner2007a,Sastre07}. These annular shaped ponds of low depth (10 to 50 cm) are mixed with a paddlewheel (see Fig.\ref{raceway3D}).  Due to their inherent nonlinear and instationary properties, where hydrodynamics and biology are strongly coupled, managing and optimizing such processes is very tricky. Carrying out experiments on raceways is both expensive and time consuming. A model is thus a key tool to help in the optimal design of the process but also in its operation. The objective of this paper is to propose a new model describing the coupling between hydrodynamics and biology within a raceway.

The first dynamic model of a microalgal raceway pond was proposed by \cite{Sukenik1987} assuming spatial homogeneity. The model was later consolidated by including time-discrete photoacclimation dynamics \cite{Sukenik1991}. In parallel, other less elaborated models where proposed \cite{Grobbelaar1990,Guterman1990}. Latter on, the coupling of biology with hydrodynamics in raceways was studied \cite{Huggins2004, Huggins2005}, in order to optimize the raceway design. In \cite{James2010}, algae growth and transport is modelled and several tests are performed in order to study for instance the effect of the water height or temperature control on the algae concentration evolution. However, those studies might not guarantee some key properties such as mass balance. We claim that the model we develop satisfies crucial mathematical properties such as the conservation of biological variables or positivity of the water height.

Our approach is the following. For the hydrodynamical part, a multilayer Saint-Venant system with variable density introduced in \cite{Audusse2010a} is applied. We use this model because we do not want to tackle the whole Navier-Stokes system for such shallow water situations, but we want to represent heterogeneity throughout the water column. This model has now been widely validated \cite{Audusse2010a, Audusse2010b, Sainte-Marie2010}. We add for our problem a specific forcing mimicking the effect of a paddlewheel. For the growth of microalgae, we utilize an improved version of the Droop model \cite{Droop1968,Droop1983,Bernard2011}. The Droop model has been widely studied and validated \cite{SciRam94,Lange1993,Bernard1995,Bernard2002}. It states that growth do not depend on the external nutrient concentration but on the internal cell quota of nutrients. The algae can indeed still grow a few days after exhaustion of the substrate thanks to their capacity to store nutrients. The enhanced Droop model \cite{Bernard2011} also takes into account the effect of light on phytoplankton. We then write the multi-layer version of the biological model, inspired from \cite{Audusse2010a}. Afterwards, we give a kinetic interpretation of the whole system, allowing the derivation of a numerical scheme that has requested properties.

The outline is organized as follows. In section \ref{secmodel} we describe and justify our system of partial differential equations. Afterwards, we derive the numerical scheme that we will use, based on a kinetic interpretation. We basically follow \cite{Audusse2010a} and add the new state variables concerning biology. In section \ref{secsim} we explain how we model the raceway with the 2D code (the two dimensions being (x,z), respectively the length of the pool and the water depth) by adding periodic conditions. We also focus on the way agitation is introduced: we add a force mimicking a paddle-wheel and then we derive its contribution in the multilayer system, in the kinetic scheme and in the discrete scheme. We show that we have relevant results. Eventually, a last section deals with a Lagrangian approach of the algae tracking that could be useful regarding the elaboration of a better environment modeling for the algae. \\

\section{The coupled model}
\label{secmodel}
We adopt and couple two continuous models, one for the representation of free surface hydrodynamics and the other for microalgae growth. We point out that it is a one way coupling: the biology is indeed advected and diffused by the water flow, but there is no retroaction of algae on the fluid. This is justified by the fact that the biological concentrations, even though greater in a raceway than they could be in an ocean or a lake, remain still much smaller than what is expected to change density or temperature of the pool (let us recall for instance that the salinity of the seawater is around 37 $g.L^{-1}$ whereas we will never reach algae concentrations more than 1 or 2 $g.L^{-1})$.

\subsection{The hydrodynamics model}

Let us introduce first the hydrodynamics model in two dimensions. It will represent a free surface flow set into motion by a paddlewheel. The paddlewheel applies a volumic force on the water: 
$$F_{wheel}(x,z,t) = F_{x}(x,z,t) \mathbf{\overrightarrow{e_{x}}} + F_{z}(x,z,t) \mathbf{\overrightarrow{e_{z}}}.$$
More details are given in section \ref{secsim} about the formulation of the force. 
We begin with the 2D hydrostatic Navier-Stokes equations with varying density, on which we plug the paddlewheel force:
\begin{align}
&\frac{\partial \rho}{\partial t} +\frac{\partial \rho u}{\partial x}+\frac{\partial \rho w}{\partial z} = 0 \label{eq:NS_2d1},\\
&\frac{\partial \rho u}{\partial t} + \frac{\partial \rho u^2}{\partial x} + \frac{\partial \rho uw}{\partial z} + \frac{\partial p}{\partial x} = \frac{\partial \Sigma_{xx}}{\partial x} + \frac{\partial \Sigma_{xz}}{\partial z} + F_{x}(x,z,t) \label{eq:NS_2d2},\\
&\frac{\partial p}{\partial z} = -\rho g + \frac{\partial \Sigma_{zx}}{\partial x} + \frac{\partial \Sigma_{zz}}{\partial z} + F_{z}(x,z,t) ,\label{eq:NS_2d3}
\end{align}
and we consider solutions of the equations  for $t>t_0, \quad x \in \R, \quad z_b(x) \leq z \leq \eta(x,t),$
where $\eta(x,t)$ represents the free surface elevation, ${\bf u}=(u,w)^T$ the velocity
vector, $p(x,z,t)$ is the pressure, $g$ the gravity acceleration and $\rho(T)$ is the water density, depending on an advected and diffused tracer $T$ basically representing the temperature and which satisfies the advection diffusion equation:

\begin{equation}
\label{tracer}
\frac{\partial \rho T}{\partial t} + \frac{\partial \rho uT}{\partial x} + \frac{\partial \rho wT}{\partial z} = \mu_{T} \frac{\partial^{2}T}{\partial x^{2}} + \mu_{T} \frac{\partial^{2}T}{\partial z^{2}}
\end{equation}

The flow height is $H = \eta - z_b$. The chosen form of the viscosity tensor is
\begin{equation*}
\Sigma_{xx} = 2 \mu \frac{\partial u}{\partial x},  \quad  \Sigma_{xz} = \mu  \frac{\partial u}{\partial z} , \quad \Sigma_{zz} = 2 \mu \frac{\partial w}{\partial z},  \quad  \Sigma_{zx} = \mu  \frac{\partial u}{\partial z} 
\label{eq:visco}
\end{equation*}
where $\mu$ is a dynamic  viscosity.

\paragraph{Boundary conditions}

The system (\ref{eq:NS_2d1})-(\ref{eq:NS_2d3}) is completed with boundary conditions. The
outward and upward unit normals to the free surface ${\bf n}_s$ and to the bottom ${\bf n}_b$ are
given by
$${\bf n}_s = \frac{1}{\sqrt{1 + \bigl(\frac{\partial \eta}{\partial x}\bigr)^2}}  \left(\begin{array}{c} -\frac{\partial \eta}{\partial x}\\ 1 \end{array} \right), \quad {\bf n}_b = \frac{1}{\sqrt{1 + \bigl(\frac{\partial z_b}{\partial x}\bigr)^2}}  \left(\begin{array}{c} -\frac{\partial z_b}{\partial x}\\ 1 \end{array} \right).$$
Let
$\Sigma_T$ be the total stress tensor with
$$\Sigma_T = - p I_d + \left(\begin{array}{cc} \Sigma_{xx} & \Sigma_{xz}\\ \Sigma_{zx} & \Sigma_{zz}\end{array}\right).$$

\paragraph{Free surface conditions}

At the free surface we have the kinematic boundary condition
\begin{equation}
\frac{\partial \eta}{\partial t} + u_s \frac{\partial \eta}{\partial x}
-w_s = 0,
\label{eq:free_surf} 
\end{equation}
where the subscript $s$ denotes the value of the considered quantity at the free surface. 
Provided that the air viscosity is negligible, the continuity of stresses at the free boundary implies\begin{equation}
\Sigma_T {\bf n}_s = -p^a {\bf n}_s, 
\label{eq:BC_h}
\end{equation}
where $p^a=p^a(x,t)$ is a given function corresponding to the atmospheric pressure. In the following, we assume $p^a=0.$

\paragraph{Bottom conditions}

The kinematic boundary condition is a classical no-penetration condition
\begin{equation}
{\bf u}_b. {\bf n}_b = 0,\quad \mbox{or}\quad u_b \frac{\partial z_b}{\partial x} - w_b = 0.
\label{eq:bottom} 
\end{equation}
For the bottom stresses we consider a wall law
\begin{equation}
{\bf t}_b. \Sigma_T {\bf n}_b = \kappa {\bf u}_b . {\bf t}_b,
\label{eq:BC_z_b}
\end{equation}
where ${\bf t}_b$ a unit vector satisfying ${\bf t}_b \cdot {\bf n}_b = 0$ and $\kappa$ is a friction coefficient.

\subsection{The biological model}

\subsubsection{Phytoplankton}
Microalgae have pigments to capture sunlight, which is turned into chemical energy during the photosynthesis process. They consume carbon dioxide, and release oxygen. Phytoplankton growth depends on the availability of carbon dioxide, sunlight, and nutrients. Required nutrients are of various types, but here we focus on inorganic nitrogen, such as nitrate, whose deprivency is known to stimulate lipid production \cite{Metting1996}. We consider the combined  influence of nitrate and light on phytoplankton growth. The nutrient limitation is taken into account by a Droop formulation \cite{Droop1968} of the growth rate. The light effect (photosynthesis and photoinhibition) is represented using a classical formulation from  \cite{Peeters1978} embedded in the model proposed by \cite{Bernard2011}.

\subsubsection{Droop model with photoadaptation}
The Droop model represents the growth of an algal biomass $C^{1}$ using a nutrient of concentration $C^{3}$ (nitrate under the form $NO_{3}$) in the medium. The concentration of particulate nitrogen (nitrogen contained in the algal biomass) is denoted $C^{2}$. Note that $C^{1}$ ($gC.m^{-3}$), $C^{2}$ ($gN.m^{-3}$)  and  $C^{3}$ ($gN.m^{-3}$) are three transportable quantities.
In order to reproduce the coupling between nutrient uptake rate $\lambda$ and growth rate  $\mu$, Droop introduced the cell quota,  $q(t,x)$, defined as the amount of internal nutrients per biomass unit: $q=\frac{C^{2}}{C^{1}}$. 
Microalgae growth rate thus depends  on the intra-cellular quota:
\begin{equation}
\mu(q) = \bar \mu (1-\frac{Q_0}{q}), 
\end{equation}
where the constant $\bar \mu$ denotes the hypothetical growth rate for infinite quota, and $Q_{0}$ is the minimum internal nutrient quota required for growth.

The nitrate uptake rate is a function of the external nitrate \cite{Dugdale67}:
\begin{equation}
\lambda(C^{3}) = \bar \lambda\frac{C^{3}}{C^{3}+K_3}, 
\end{equation}
where $K_{3}$ is the half saturation constant and $\bar \lambda$ the maximum uptake rate.

Finally, we will take into account both respiration and mortality, which are represented by a constant loss of the biomass with a factor $R$.

In line with \cite{Bernard2011}, we modify this classical model in order to capture light and space variations.

\paragraph{Light along the raceway depth.}
The previous model only included nutrient limited growth. It can be improved by introducing a new data in the system: the light intensity, which will depend on the quantity of water and biomass above the algae. Light is indeed attenuated by the chlorophyll concentration. This concentration can be linked to nitrogen through
$$Chl = \gamma(I^*)C^{2}$$
with $$ \gamma(I^*) = \frac{k_{I^{*}}}{I^{*}+k_{I^{*}}},$$ 
where $k_{I^{*}}$ is a constant,  $I^*$ is the average light in the water column the day before(space and time average). $\gamma(I^*)$ is presumed constant over the day.

Let us assume that light intensity hitting the water surface is of the form
\begin{equation}
I_0 = I_0^{max}\max(0,\sin(2\pi t)),
\end{equation}
Then the intensity at depth $z$ can be described by
\begin{equation}
I(z) = I_0e^{-\psi(C^{2},I^*,z)},
\end{equation}
where
\begin{align}
\psi(C^{2},I^*,z) &= \int_0^z{(a\gamma(I^*)C^{2}(z)+b)dz},\\
\end{align}
and $I_{0}^{max}$ is a constant representing the maximum light intensity, $a$ and $b$ are also given constants.

The growth rate can then be computed to take into account light intensity, using Peeters and Eilers formalism \cite{Peeters1978,Han01}
$$\mu(q,I) =\tilde{\mu}\frac{I}{I + K_{sI} + \frac{I^2}{K_{iI}}}(1-\frac{Q_0}{q}),$$
with $\tilde{\mu}$, $K_{sI}$, $K_{iI}$ three given constants derived from dedicated experiments.
Finally, a  down regulation by the internal quota of the uptake rate must be included to avoid infinite substrate($NO_{3}$) uptake in the dark
$$\lambda(C^{3},q) =\overline{\lambda}\frac{C^{3}}{C^{3}+K_3}(1-\frac{q}{Q_l})$$
where $Q_l$ is the maximum achievable quota.

Adding advection and diffusion, the biological system writes in the end
\begin{align}
&\frac{\partial \rho C^{1}}{\partial t} + \frac{\partial \rho uC^{1}}{\partial x} + \frac{\partial \rho w C^{1}}{\partial z} =
\mu_{C^{1} } \left( \frac{\partial^{2}C^{1}}{\partial x^{2}} + \frac{\partial^{2}C^{1}}{\partial z}\right)+ \rho(\mu(q,I)C^{1} - R C^{1}),\label{eqC1}\\
&\frac{\partial \rho C^{2}}{\partial t} +  \frac{\partial \rho uC^{2}}{\partial x} + \frac{\partial \rho w C^{2}}{\partial z}  =
\mu_{C^{2} } \left( \frac{\partial^{2}C^{2}}{\partial x^{2}} + \frac{\partial^{2}C^{2}}{\partial z}\right) + \rho(\lambda(C^{3},q) C^{1} - R C^{2}),\label{eqC2}\\
&\frac{\partial \rho C^{3}}{\partial t} + \frac{\partial \rho uC^{3}}{\partial x} + \frac{\partial \rho w C^{3}}{\partial z} =
\mu_{C^{3} } \left( \frac{\partial^{2}C^{3}}{\partial x^{2}} + \frac{\partial^{2}C^{3}}{\partial z}\right) - \rho\lambda(C^{3},q) C^{1},\label{eqC3},
\end{align}
where $q = \frac{C^{2}}{C^{1}}$, $\mu_{C^{1} }, \mu_{C^{2} }, \mu_{C^{3} }$ are diffusion coefficients and ($u$, $w$) are the fluid velocities along $x$ and $z$ direction.

\section{Multilayer model, kinetic interpretation and numerical scheme}
\label{secnumscheme}

\subsection{Vertical space discretization: the multilayer model}
The next step is the vertical discretization of (\ref{eq:NS_2d1}-\ref{eq:NS_2d3}) and (\ref{eqC1}-\ref{eqC3}) in order to obtain a multilayer system. Such a derivation has already been performed in \cite{Audusse2010a, Audusse2010b}. But in our case we have a source term representing the paddlewheel effect in equations (\ref{eq:NS_2d2}) and (\ref{eq:NS_2d3}). Since the resulting vertically discretized equations are not straightforward from those two previous papers, we detail it here. However, for the sake of simplicity, we choose to omit the viscosity terms. Moreover, we will use the classical equation of state relating the density and the tracer identified here with the temperature
\begin{equation}
\rho(T) = \rho_{0}\left( 1 - \alpha(T-T_{0})^{2}\right),
\end{equation}
with $T_{0} = 4^\circ$ C, $\alpha = 6.63 10^{-6} $C$^{2}$ and $\rho_{0} = 10^{3}$ kg.m$^{-3}$.
In the range of temperatures that we plan to use, it is easy to check that density variations are small. Hence we use the Boussinesq assumption which states that those density variations are taken into account in the gravitational force only. In any other equation the density is supposed to be $\rho_{0}$. Eventually, we end up with the hydrodynamic system
\begin{align}
&\frac{\partial  u}{\partial x}+\frac{\partial  w}{\partial z} = 0 \label{eq:NS_2d1Boussi},\\
&\frac{\partial  u}{\partial t} + \frac{\partial  u^2}{\partial x} + \frac{\partial  uw}{\partial z} +\frac{1}{\rho_{0}} \frac{\partial p}{\partial x} = \frac{1}{\rho_{0}} F_{x}(x,z,t)\label{eq:NS_2d2Boussi},\\
&\frac{\partial p}{\partial z} = -\rho g + F_{z}(x,z,t)\label{eq:NS_2d3Boussi},\\
&\frac{\partial T}{\partial t} + \frac{\partial  uT}{\partial x} + \frac{\partial  wT}{\partial z} = 0 \label{eq:NS_2d4Boussi}
\end{align}
and the biological system
\begin{align}
&\frac{\partial C^{1}}{\partial t} + \frac{\partial uC^{1}}{\partial x} + \frac{\partial  w C^{1}}{\partial z} =
\mu(q,I)C^{1} - R C^{1},\label{eqC1Boussi}\\
&\frac{\partial C^{2}}{\partial t} +  \frac{\partial uC^{2}}{\partial x} + \frac{\partial w C^{2}}{\partial z}  =
\lambda(C^{3},q) C^{1} - R C^{2},\label{eqC2Boussi}\\
&\frac{\partial C^{3}}{\partial t} + \frac{\partial uC^{3}}{\partial x} + \frac{\partial w C^{3}}{\partial z} =
 -\lambda(C^{3},q) C^{1},\label{eqC3Boussi}
\end{align}
with $q = \frac{C^{2}}{C^{1}}$.

The process to obtain the multilayer system is described below. It is basically a Galerkin approximation of the variables followed by a vertical integration of the equations. The interval $[z_b,\eta]$ is divided into $N$ layers $\{L_\alpha\}_{\alpha\in\{1,\ldots,N\}}$ of thickness $l_\alpha H(x,t)$ where each layer $L_\alpha$ corresponds to the points satisfying $z \in L_\alpha(x,t) = [z_{\alpha-1/2},z_{\alpha+1/2}]$ with
\begin{equation}
\left\{\begin{array}{l}
z_{\alpha+1/2}(x,t) = z_b(x,t) + \sum_{j=1}^\alpha l_j H(x,t),\\
h_\alpha(x,t) = z_{\alpha+1/2}(x,t) - z_{\alpha-1/2}(x,t)=l_\alpha H(x,t),
 \quad \alpha\in [0,\ldots,N]
\end{array}\right.
\label{eq:layer}
\end{equation}
 with $l_j>0, \quad \sum_{j=1}^N l_j=1$. 

Now let us consider the space $\Bbb{P}_{0,H}^{N,t}$ of piecewise constant functions defined by
$$\Bbb{P}_{0,H}^{N,t} = \left\{ \Bbb{I}_{z \in L_\alpha(x,t)}(z),\quad \alpha\in\{1,\ldots,N\}\right\},$$
where $\Bbb{I}_{z \in L_\alpha(x,t)}(z)$ is the characteristic function of the interval $L_\alpha(x,t)$. Using this formalism, the projection of $u$, $w$ and $T$ onto $\Bbb{P}_{0,H}^{N,t}$ is a piecewise constant function defined by
\begin{equation}
X^{N}(x,z,\{z_\alpha\},t)  = \sum_{\alpha=1}^N
\Bbb{I}_{[z_{\alpha-1/2},z_{\alpha+1/2}]}(z)X_\alpha(x,t),
\label{eq:ulayer}
\end{equation}
for $X \in (u,w,T)$. The density $\rho=\rho(T)$ inherits a discretization from the previous relation with
\begin{equation}
\rho^{N}(x,z,\{z_\alpha\},t)  = \sum_{\alpha=1}^N
\Bbb{I}_{[z_{\alpha-1/2},z_{\alpha+1/2}]}(z)\rho(T_\alpha(x,t)).
\label{eq:rholayer}
\end{equation}

We have the following result.
\begin{prop}
The weak formulation of Eqs.~(\ref{eq:NS_2d1Boussi})-(\ref{eq:NS_2d4Boussi}) on $\Bbb{P}_{0,H}^{N,t}$ leads to a system of the form

\begin{alignat}{1}
\sum_{\alpha=1}^N \frac{\partial l_\alpha H  }{\partial t }   + 
\sum_{\alpha=1}^N \frac{\partial l_\alpha H u_\alpha }{\partial x } &= 0,
\label{eq:H}\\
\frac{\partial  h_\alpha u_\alpha}{\partial t } +
\frac{\partial }{\partial x }\left(h_\alpha u^2_\alpha +
h_\alpha p_\alpha\right) &= u_{\alpha+1/2}G_{\alpha+1/2} - u_{\alpha-1/2}G_{\alpha-1/2}\notag \\
                                            & +\frac{1}{\rho_{0}}\left(\frac{\partial z_{\alpha+1/2}}{\partial x} p_{\alpha+1/2} - \frac{\partial z_{\alpha-1/2}}{\partial x} p_{\alpha-1/2}\notag \right) \\
                                            &+\frac{1}{\rho_{0}} \int_{z_{\alpha-1/2}}^{z_{\alpha+1/2}}{F_x(x,z,t)dz},
\label{eq:eq4}\\
\frac{\partial h_{\alpha}T_{\alpha}}{\partial t} + \frac{\partial}{\partial x}\left(h_{\alpha}u_{\alpha}T_{\alpha}\right) &= T_{\alpha+1/2}G_{\alpha+1/2} - T_{\alpha-1/2}G_{\alpha-1/2},\\
&\alpha\in [1,\ldots,N].\nonumber
\end{alignat}

\label{prop:mc_simple}
with 
\begin{align}
&G_{\alpha+1/2} = \frac{\partial z_{\alpha+1/2}}{\partial t} + u_{\alpha+1/2}\frac{\partial z_{\alpha+1/2}}{\partial x}-w_{\alpha+1/2} \label{eq:massexchange}\\
&G_{1/2} = G_{N+1/2} = 0.\label{eq:massexchangelim}
\end{align}
The definitions of $p_\alpha$, $p_{\alpha+1/2}$, $u_{\alpha+1/2}$, $T_{\alpha+1/2}$ are given in the following proof. 
\end{prop}

\begin{proof}
Since the demonstration for part of this proposal can be found in \cite{Audusse2010a} and \cite{Audusse2010b} , we will not detail it here. However, we will focus on the integration of the agitation terms in this multilayer system.

\paragraph{The horizontal term}
The horizontal term, in the x-projection of the momentum equation will be handled simply: since we plan to use an expression of the force that can be integrated analytically (see \ref{paddle}), we just add the integrated force in the right-hand-side of (\ref{eq:eq4})
\begin{align}
+ \int_{z_{\alpha-1/2}}^{z_{\alpha+1/2}}{F_x(x,z,t)dz},\\
& \alpha\in [1,\ldots,N]. \label{eq:Q}\nonumber
\end{align} 

\paragraph{The vertical term}

The handling of this term will require more steps. First of all, let us recall that in the situation where no additional source terms are present, (\ref{eq:NS_2d3Boussi}) is used to compute the piecewise continuous state variable $p$, which is then introduced in (\ref{eq:NS_2d2Boussi}) (see \cite{Audusse2010a}). We basically proceed the same way.\\

From (\ref{eq:NS_2d3Boussi}), we get:
\begin{equation}
p(x,z,t) = g\int_{z}^{\eta}{\rho(x,z',t)}dz' - \int_{z}^{\eta}{F_{z}(x,z',t)}dz'\label{eq:p_agit}
\end{equation}

If we project it on $\Bbb{P}_{0,H}^{N,t}$, we get for a $z$ in layer $\alpha$
\begin{equation}
p(x,z,t) = g\left( \sum_{j=\alpha +1}^{N}{\rho_j h_j} + \rho_{\alpha}(z_{\alpha + 1/2} - z)\right) - \int_{z}^{\eta}{F_{z}(x,z',t)dz'}.
\label{eq:p_agit_expand}
\end{equation}
For the vertical integration of (\ref{eq:NS_2d2Boussi}) we need to compute:

\begin{equation}
\int_{z_{\alpha -1/2}}^{z_{\alpha +1/2}}{\frac{\partial p}{\partial x} dz} = \frac{\partial h_{\alpha}p_{\alpha}}{\partial x} - \frac{\partial z_{\alpha +1/2}}{\partial x}p_{\alpha + 1/2} + \frac{\partial z_{\alpha -1/2}}{\partial x}p_{\alpha - 1/2}.
\label{dpdx_bis}
\end{equation}

Since 
\begin{equation}
p_{\alpha} =\frac{1}{h_\alpha}\displaystyle  \int_{z_{\alpha-1/2}}^{z_{\alpha+1/2}}p(x,z,t)dz
,\qquad p_{\alpha+1/2} = p(x,z_{\alpha+1/2},t),
\label{eq:palpha_bis}
\end{equation}
we see that the vertical agitation force has an effect in $p$ through (\ref{eq:NS_2d2Boussi}). Let us precise the expressions of $p_{\alpha}(x,t)$, $p_{\alpha + 1/2}(x,t)$ and $p_{\alpha -1/2}(x,t)$.
\begin{align}
p_{\alpha}(x,t) &= \frac{1}{h_\alpha}\displaystyle  \int_{z_{\alpha-1/2}}^{z_{\alpha+1/2}}p(x,z,t)dz  \nonumber \\
			&=g\left( \frac{\rho_{\alpha}h_{\alpha}}{2} + \sum_{j=\alpha +1}^{N}{\rho_j h_j}\right) - \frac{1}{h_{\alpha}}\int_{z_{\alpha -1/2}}^{z_{\alpha +1/2}}{\int_{z}^{\eta}{F_z(x,z',t)dz'}dz}\label{palphadef},\\
p_{\alpha + 1/2}(x,t) &=  g \sum_{j=\alpha +1}^{N}{\rho_j h_j} - \int_{z_{\alpha + 1/2}}^{\eta}{F_z(x,z',t)dz'},\\
p_{\alpha -1/2}(x,t) &= g \sum_{j=\alpha}^{N}{\rho_j h_j} - \int_{z_{\alpha -1/2}}^{\eta}{F_z(x,z',t)dz'}.\end{align}
The velocities $u_{\alpha+1/2}, \, \alpha=1, ..., N-1$ are obtained using an upwinding with respect to the direction of the mass exchange:
\begin{equation}
u_{\alpha+1/2} =  \left\{
    \begin{array}{ll}
        u_{\alpha} & \mbox{if } G_{\alpha+1/2} \ge 0 \\
        u_{\alpha+1} & \mbox{if } G_{\alpha+1/2} < 0.
    \end{array}
\right.
\label{ualphaplus}
\end{equation}
We proceed in the same way for the tracer:
\begin{equation}
T_{\alpha+1/2} =  \left\{
    \begin{array}{ll}
        T_{\alpha} & \mbox{if } G_{\alpha+1/2} \ge 0 \\
        T_{\alpha+1} & \mbox{if } G_{\alpha+1/2} < 0.
    \end{array}
\right.
\label{Talphaplus}
\end{equation}
\end{proof}

In Prop.~\ref{prop:mc_simple} the vertical velocity $w$ no more
appears, but we can derive  relations for the discrete layer values of
this variable by performing the Galerkin approximation of the continuity
equation (\ref{eq:NS_2d1Boussi}) multiplied by $z$. This  leads to
\begin{align}
\frac{\partial}{\partial t }\left(\frac{z_{\alpha+1/2}^2 -
  z_{\alpha-1/2}^2}{2} \right) + \frac{\partial}{\partial
  x}\left( \frac{z_{\alpha+1/2}^2 -  z_{\alpha-1/2}^2}{2}
   u_\alpha\right) &= h_\alpha  w_\alpha + z_{\alpha+1/2}G_{\alpha+1/2} - z_{\alpha-1/2}G_{\alpha-1/2},
\label{eq:w_mc}
\end{align}
where the $w_\alpha$, $\alpha =1,\ldots,N$, are the components of the
Galerkin approximation of $w$ on $\Bbb{P}_{0,H}^{N,t}$, see (\ref{eq:ulayer}). 
Since all the quantities except $w_\alpha$ appearing in Eq.~(\ref{eq:w_mc}) 
are already defined by (\ref{eq:H}), (\ref{eq:eq4}), (\ref{eq:massexchange}), (\ref{eq:massexchangelim}),
relation (\ref{eq:w_mc}) allows obtaining the values  $w_\alpha$
by post-processing.
Note that we use the relation (\ref{eq:w_mc})
rather than the divergence free condition 
for stability purposes.
We refer the reader to \cite{Sainte-Marie2010,Audusse2010a} for more details.

\noindent \begin{prop} Multilayer version of the Droop model with photoacclimation. The weak formulation of  Eqs.~(\ref{eqC1Boussi})-(\ref{eqC3Boussi}) on $\Bbb{P}_{0,H}^{N,t}$ leads to a system of the form
\begin{align}
\frac{\partial h_\alpha C_{ \alpha}^{1}}{\partial t } +
\displaystyle\frac{\partial }{\partial x}\left( h_\alpha
  C_{\alpha}^{1} u_\alpha\right) &= C_{ \alpha+1/2}^{1} G_{\alpha+1/2} -
C_{\alpha-1/2}^{1} G_{\alpha-1/2} \nonumber\\
& + h_\alpha(\mu(q_\alpha, I_\alpha)C_{\alpha}^{1} - R C_{ \alpha}^{1}),\label{eq:eqX}\\
\frac{\partial  h_\alpha C_{ \alpha}^{2}}{\partial t } +
\displaystyle\frac{\partial }{\partial x}\left( h_\alpha
  C_{ \alpha}^{2} u_\alpha\right) &= C_{\alpha+1/2}^{2} G_{\alpha+1/2} -
C_{\alpha-1/2}^{2} G_{\alpha-1/2}\nonumber\\
& +  h_\alpha(\lambda(C_{ \alpha}^{3},q_\alpha) C_{ \alpha}^{1} - R C_{ \alpha}^{2}),\label{eq:eqN}\\
\frac{\partial h_\alpha C_{ \alpha}^{3}}{\partial t } +
\displaystyle\frac{\partial }{\partial x}\left( h_\alpha
  C_{ \alpha}^{3} u_\alpha\right) &= C_{ \alpha+1/2}^{3} G_{\alpha+1/2} -
C_{\alpha-1/2}^{3} G_{\alpha-1/2}\nonumber\\
&  -  h_\alpha\lambda(C_{ \alpha}^{3},q_\alpha) C_{ \alpha}^{1} ,\label{eq:eqS}\\
&\hspace{5cm} \alpha\in [1,\ldots,N],\nonumber
\end{align}
with $q_{\alpha} = \frac{C_{\alpha}^{2}}{C_{\alpha}^{1}}$ and $C_{\alpha+1/2}^{j}$, $j=1\ldots 3$ defined through the same upwinding as $u_{\alpha+1/2}$ and $T_{\alpha+1/2}$ (see \ref{ualphaplus}, \ref{Talphaplus}).
\label{prop:mldroop}
\end{prop}

\begin{proof} As for Eqs.~(\ref{eq:NS_2d1Boussi})-(\ref{eq:NS_2d4Boussi}), we use a Galerkin approximation of the biological variables on $\Bbb{P}_{0,H}^{N,t}$. The Galerkin approximation on $\Bbb{P}_{0,H}^{N,t}$ allows to write 
\begin{equation}
Y^{N}(x,z,\{z_\alpha\},t)  = \sum_{\alpha=1}^N
1_{[z_{\alpha-1/2},z_{\alpha+1/2}]}(z)Y_\alpha(x,t),
\label{eq:ulayer2}
\end{equation}
for $Y \in (C^{1},C^{2},C^{3})$.
\noindent Then we perform an integration of Eqs.~(\ref{eqC1Boussi})-(\ref{eqC3Boussi}) over the layer $\alpha$. Let us do it term by term.
\begin{equation}
\int_{z_{\alpha-1/2}}^{z_{\alpha+1/2}}{\frac{\partial  Y}{\partial t }dz} = \frac{\partial h_{\alpha}Y_{\alpha}}{\partial t} - (Y_{\alpha+1/2}\frac{\partial z_{\alpha+1/2}}{\partial t} - Y_{\alpha-1/2}\frac{\partial z_{\alpha-1/2}}{\partial t})
\end{equation}

\noindent For the transport terms it yields:
\begin{align}
\int_{z_{\alpha-1/2}}^{z_{\alpha+1/2}}{\frac{\partial  uY}{\partial x }dz} &= \frac{\partial u_{\alpha}h_{\alpha}Y_{\alpha}}{\partial x} - (u_{\alpha+1/2}Y_{\alpha+1/2}\frac{\partial z_{\alpha+1/2}}{\partial x} - u_{\alpha-1/2}Y_{\alpha-1/2}\frac{\partial z_{\alpha-1/2}}{\partial x})\\
\int_{z_{\alpha-1/2}}^{z_{\alpha+1/2}}{\frac{\partial  wY}{\partial z }dz} &= w_{\alpha+1/2}Y_{\alpha+1/2} - w_{\alpha-1/2}Y_{\alpha-1/2}
\end{align}

\noindent For the reaction terms we get:
\begin{align}
\int_{z_{\alpha-1/2}}^{z_{\alpha+1/2}}{(\mu(q,I)
C^{1} - R C^{1})dz} &=h_{\alpha}(\mu(q_{\alpha},I_{\alpha})
C^{1}_{\alpha} - R C^{1}_{\alpha}),\\
\int_{z_{\alpha-1/2}}^{z_{\alpha+1/2}}{(\lambda(C^{3},q) C^{1} - R C^{2})dz}&= h_{\alpha}(\lambda(C^{3}_{\alpha},q_{\alpha}) C^{1}_{\alpha} - R C^{2}_{\alpha}),\\
\int_{z_{\alpha-1/2}}^{z_{\alpha+1/2}}{-\lambda(C^{3},q) X dz} &= -h_{\alpha}(\lambda(C^{3}_{\alpha},q_{\alpha}) C^{1}_{\alpha}),
\end{align}
where $$q = \frac{C^{2}}{C^{1}} \mbox{ and } q_{\alpha} = \frac{C^{2}_{\alpha}}{C^{1}_{\alpha}}.$$

\noindent Using the notation (\ref{eq:massexchange}) and (\ref{eq:massexchangelim}) we recover Eqs.~(\ref{eq:eqX})-(\ref{eq:eqS}).
\end{proof}
\subsection{Kinetic interpretation}

The kinetic approach consists in linking the behaviour of some macroscopic fluid systems - Euler or Navier-Stokes equations, Saint-Venant system - with Boltzmann type kinetic equations. Boltzmann equation was first introduced in gas dynamics. It represents the evolution of a density of particles in a gas. Kinetic schemes have been widely used for the resolution of Euler equations \cite{Khobalatte1992, Bourgat1994}. Given the analogy between Euler and Saint-Venant equations, recent work has been carried out to adapt those schemes to the shallow water systems (\cite{Audusse2004a}). The first step is the introduction of fictitious particles, the definition of a density of particles and the equation governing its evolution.

\subsection{The kinetic equations}

The process to obtain the kinetic interpretation of the multilayer model with varying density is very similar to the one used in \cite{Audusse2010a}. For that reason we will only detail the kinetic interpretation of the biological system. For a given layer $\alpha$, a distribution function $M_\alpha(x,t,\xi)$ of fictitious particles with microscopic velocity $\xi$ is introduced to obtain a linear kinetic equation equivalent to the macroscopic model.\\

\subsection{Kinetic model derivation}

Let us introduce a real function $\chi$ defined on
$\mathbb{R}$, compactly supported and which have the following properties
\begin{equation}
\left\{\begin{array}{l}
\chi(-w) = \chi(w) \geq 0\\
\int_{\mathbb{R}} \chi(w)\ dw = \int_{\mathbb{R}} w^2\chi(w)\ dw = 1.
\end{array}\right.
\label{eq:chi1}
\end{equation}
Now let us construct a density of particles $M_\alpha(x,t,\xi)$
defined by a Gibbs equilibrium: the microscopic density of particles
present at time $t$, in the layer $\alpha$, at the abscissa $x$ and with
velocity $\xi$ given by
\begin{equation}
M_\alpha = \frac{ h_\alpha(x,t)}{c_\alpha} \chi\left(\frac{\xi -
    u_\alpha(x,t)}{c_\alpha}\right),
\label{eq:Malpha}
\end{equation}
with
$$c_\alpha^2 = p_\alpha,$$
and $p_\alpha$ defined by (\ref{palphadef}).\\

Likewise, we define $N_{\alpha+1/2}(x,t,\xi)$ by
\begin{equation}
 \hspace*{-1.2cm} N_{\alpha+1/2}(x,t,\xi) = G_{\alpha+1/2}(x,t) \ \delta \left(\xi -
  u_{\alpha+1/2}(x,t)\right),\label{eq:N}
\end{equation}
for $\alpha=0,\ldots,N$ and where $\delta$ denotes the Dirac distribution. \\

The quantities $G_{\alpha+1/2}$, $0\leq\alpha\leq N$ represent the
mass exchanges between layers $\alpha$ and $\alpha+1$, they are defined in
(\ref{eq:massexchange}) and satisfy the conditions (\ref{eq:massexchangelim}), so $N_{1/2}$ and $N_{N+1/2}$ also satisfy 
\begin{equation}
N_{1/2}(x,t,\xi) = N_{N+1/2}(x,t,\xi) = 0.
\label{eq:Nlim}
\end{equation}
For the Droop variables, we have the equilibria
\begin{eqnarray}
&& U^j_\alpha(x,t,\xi) = C^{j}_\alpha(x,t) M_\alpha(x,t,\xi),\qquad \alpha=1,\ldots,N, \qquad j=1,\ldots,3 \label{eq:Uialpha}\\
&& V^j_{\alpha+1/2}(x,t,\xi) = C^{j}_{\alpha+1/2}(x,t) N_{\alpha+1/2}(x,t,\xi),\qquad \alpha=0,\ldots,N, \qquad j=1,\ldots,3.  \label{eq:Vialpha}
\end{eqnarray}
With the previous definitions we write a kinetic representation of the
multilayer system \textbf{without biological reaction terms}
and we have the following proposition:
\begin{prop}
The functions $(C^{j})$ are strong solutions of the system (\ref{eq:eqX})
-(\ref{eq:eqS}) without reaction terms if and only if the set of
equilibria $\{U^{j}_\alpha(x,t,\xi)\}_{\alpha=1}^N$
are solutions of the kinetic equations
\begin{equation}
 \frac{\partial U^j_\alpha}{\partial t} + \xi \frac{\partial U^j_\alpha}{\partial x} 
 - V^j_{\alpha+1/2} + V^j_{\alpha-1/2} = Q_{U^{j}_\alpha},\label{eq:gibbsC}
\end{equation}
for $\alpha=1,\ldots,N$, $j = 1\ldots 3$ with $\{N_{\alpha+1/2}(x,t,\xi), V^j_{\alpha+1/2}(x,t,\xi)\}_{\alpha=0}^N$ satisfying
 (\ref{eq:N})-(\ref{eq:Vialpha}). 

The quantities $Q_{U^j_\alpha} = Q_{U^j_\alpha}(x,t,\xi)$ are ``collision terms''  equal to zero at the
macroscopic level i.e. which satisfy for a.e. values of $(x,t)$
\begin{equation}
\int_{\mathbb{R}} Q_{U^j_\alpha} d\xi = 0,\quad \int_{\mathbb{R}} \xi
Q_{U^j_\alpha} d\xi =0,\quad\mbox{and}\quad\quad \int_{\mathbb{R}}
\xi^{2}Q_{U^j_\alpha} d\xi = 0.
\label{eq:collision}
\end{equation}
\label{prop:kinetic_sv_mc}
\end{prop}
\begin{proof}
Using the definitions (\ref{eq:Malpha}),(\ref{eq:Uialpha}) and the properties of the
function $\chi$, we have
\begin{equation}
l_\alpha H C^{j}_\alpha = \int_{\R}   U^j_\alpha (x,t,\xi) d\xi,\qquad
l_\alpha H C^{j}_\alpha u_\alpha = \int_{\R}   \xi U^j_\alpha (x,t,\xi) d\xi. \label{eq:intUi}\\
\end{equation}
From the
definition (\ref{eq:N}) of $N_{\alpha+1/2}$ we also have
\begin{equation}
\int_{\R} N_{\alpha+1/2}(x,t,\xi) d\xi = G_{\alpha+1/2},\label{eq:intN}
\end{equation}
A simple integration in $\xi$ of the equations
(\ref{eq:gibbsC}), always using (\ref{eq:collision}), gives
 the biological equations (\ref{eq:eqX})-(\ref{eq:eqS}). 
 \end{proof}

\subsection{The numerical scheme}
\label{numscheme}
Adding the viscous terms to the multilayer system obtained in Prop.~\ref{prop:mldroop} we end up with a system of the form
\begin{equation}
\frac{\partial X}{\partial t} + \frac{\partial F(X)}{\partial x} =
S_e(X) +  S_{v,f}(X) + S_{bio}(X),
\label{eq:glo}
\end{equation}
with $X=\left(k_{1}^{1},\ldots,k_{N}^{1},k_{1}^{2}\ldots k_{N}^{2}, k_{1}^{3}\ldots k_{N}^{3}\right)^T$ and  $k_{\alpha}^{1}= l_{\alpha} H
C^{1}_{\alpha}$, $k_{\alpha}^{2}= l_{\alpha} H
C^{2}_{\alpha}$, $k_{\alpha}^{3}= l_{\alpha} H
C^{3}_{\alpha}$. We denote $F(X)$ the flux of the conservative part,
$S_e(X)$, $S_{v,f}(X)$ and $S_{bio}(X)$ the source terms, respectively the
mass transfer,  the viscous and
friction effects and the biological reaction terms.\\

We introduce a $3N \times 3N$ matrix ${\cal K(\xi)}$ 
defined by ${\cal K}_{i,j} =\delta_{i,j}$ for $i,j=1,\ldots,N$ with $\delta_{i,j}$ the Kronecker symbol. Then, using Prop. \ref{prop:kinetic_sv_mc}, we can write
\begin{eqnarray}
X= \int_{\xi} {\cal K( \xi)} 
\left(\begin{array}{c}  U^{1}(\xi) \\  U^{2}(\xi)\\  U^{3}(\xi)\end{array}\right) d \xi, 
\quad F(X) = \int_{\xi}  \xi {\cal K( \xi)} 
\left(\begin{array}{c}  U^{1}(\xi) \\  U^{2}(\xi)\\  U^{3}(\xi) \end{array}\right) d \xi,
\label{eq:fx}\\
S_e(X) = \int_{\xi} {\cal K( \xi)} 
\left(\begin{array}{c}  V^{1}(\xi) \\  V^{2}(\xi)\\  V^{3}(\xi) \end{array}\right) d \xi, 
\label{eq:fxx}
\end{eqnarray}
with $U^{j}(\xi)=(U_{1}^{j}(\xi),\ldots,U_{N}^{j}(\xi))^T$,  and
$$
V(\xi)^{j}=\left(
\begin{array}{c}
V_{3/2}^{j}(\xi)-V_{1/2}^{j}(\xi)\\
\vdots\\
V_{N+1/2}^{j}(\xi)-V_{N-1/2}^{j}(\xi)
\end{array}
\right), \quad \forall j \in {1..3}.
$$ 
We refer to \cite{Audusse2010b} for the computation of $S_{v,f}(X)$.

To approximate the solution of (\ref{eq:glo}) we use a finite volume framework.
We assume that the computational domain is discretized by $I$ nodes $x_i$.
We denote $C_i$ the cell of length $\Delta
x_i=x_{i+1/2}-x_{i-1/2}$ with $x_{i+1/2}=(x_i+x_{i+1})/2$. For the time discretization, 
we denote $t^n = \sum_{k < n} \Delta t^k$ where the time steps $\Delta t^k$ will be
precised later through a CFL condition. 
We denote
$$X^n_{i}=\left(k^{1,n}_{1,i},\ldots,k^{1,n}_{N,i}, k^{2,n}_{1,i},\ldots,k^{2,n}_{N,i}, k^{3,n}_{1,i},\ldots,k^{3,n}_{N,i}\right)^T$$
the approximate solution at time $t^n$ on the cell $C_i$ with $k^{j,n}_{\alpha,i}= l_{\alpha} H^n_i C^{j,n}_{\alpha,i}$.

\subsection{Time splitting}

For the time discretization, we apply a time splitting to the equation (\ref{eq:glo}) and we write
\begin{eqnarray}
&&\frac{{\tilde X}^{n+1}-X^{n}}{\Delta t^n} + \frac{\partial
  F(X^n)}{\partial x} = S_e(X^n,{\tilde X}^{n+1}) +  S_{bio} (X^{n}),
\label{eq:glo1}\\
&&\frac{X^{n+1}-{\tilde X}^{n+1}}{\Delta t^n} - S_{v,f} (X^{n},X^{n+1})=0.
\label{eq:glo2}
\end{eqnarray}
The conservative part of (\ref{eq:glo1}) is computed by an explicit kinetic scheme. The mass exchange terms are deduced from the kinetic interpretation. The biological reaction terms are not included in the kinetic interpretation but simply deduced from quadrature formulas. Eventually, the viscous and friction terms $S_{v,f}$ in (\ref{eq:glo2}) do not depend on the fluid density $\rho$. Thus, their vertical discretization and their numerical treatment do not differ from earlier works of the authors \cite{Audusse2010b}. Due to potential dissipative effects, a semi-implicit scheme is adopted in the second step for reasons of stability.

\subsection{Discrete kinetic equation}

Starting from a piecewise constant approximation of the initial data, the general form of a finite volume discretization of system (\ref{eq:glo1}) is
\begin{equation}
{\tilde X}^{n+1}_{i}-X^{n}_{i} + \sigma_i^n
\left[F^n_{i+1/2}- F^n_{i-1/2}\right] = \Delta t^n {\cal S}_{e,i}^{n+1/2} + \Delta t^n {\cal S}_{bio}^{n},
\label{eq:fv}
\end{equation}
where 
$\sigma^n_i=\Delta t^n/\Delta x_i$ is the ratio between space and time steps and the numerical flux $F^n_{i+1/2}$ is an approximation of the exact flux estimated at point $x_{i+1/2}$.\\

In order to find a good expression of the numerical fluxes, we need to make an incursion in the microscopic scale. Therefore we denote the discrete particle density at time $n$, cell $i$ in the following manner:

$$M_{\alpha,i}^n (\xi)=  l_{\alpha}\frac{H^n_{i}}{c^n_{\alpha,i}} \chi\left(\frac{\xi - u^n_{\alpha,i}}{c^n_{\alpha,i}}\right), 
\qquad\mbox{with}\ c^n_{\alpha,i}= \sqrt{p^n_{\alpha,i}}$$
and following (\ref{palphadef})
$$p^n_{\alpha,i} = g \left(\frac{\rho^n_{\alpha,i} l_\alpha H^n_{i}}{2} + \sum_{j=\alpha+1}^N
\rho^n_{j,i} l_j H^n_{i} \right)- \frac{1}{h_{\alpha}}\int_{z_{\alpha -1/2}}^{z_{\alpha +1/2}}{\int_{z}^{\eta}{F_z(x_{i},z',t^{n})dz'}dz}.$$
Then the equation (\ref{eq:gibbsC}) is discretized for each $\alpha$ by applying a simple upwind scheme
\begin{eqnarray}
& &\hspace*{-1cm} g_{\alpha,i}^{j,n+1} (\xi) = U_{\alpha,i}^{j,n} (\xi) - \xi\sigma^n_i
\left(U_{\alpha,i+1/2}^{j,n} (\xi) - U_{\alpha,i-1/2}^{j,n} (\xi)\right)
+\nonumber\\
& & \quad\Delta t^n \left( V_{\alpha+1/2,i}^{j,n+1/2} (\xi) - V_{\alpha-1/2,i}^{j,n+1/2}(\xi)\right),
\label{eq:cindisg}
\end{eqnarray}
where
$$ U_{\alpha,i+1/2}^{j,n} =
\left\{\begin{array}{ll}
U_{\alpha,i}^{j,n} & \mbox{if } \xi \geq 0\\
U_{\alpha,i+1}^{j,n} & \mbox{if } \xi < 0.
\end{array}\right.$$
The quantity $g_{\alpha,i}^{j,n+1}$ is not an equilibrium 
but if we set
\begin{equation}
l_\alpha H^{n+1}_{i} C_{\alpha,i}^{j,n+1} =\int_{\R} g^{j,n+1}_{\alpha,i} (\xi) d\xi.
\label{eq:Hdis1}
\end{equation}
we recover the macroscopic quantities at time $t^{n+1}$.

\subsection{Numerical flux of the finite volume scheme}

In this section, we give some details for the computation of the fluxes introduced in 
the discrete equation (\ref{eq:fv}).
If we denote
\begin{equation}
F^n_{i+1/2} = F(X^n_i,X^n_{i+1}) = F^+(X^n_i) + F^-(X^n_{i+1}),
\label{eq:fluxes}
\end{equation}
following (\ref{eq:fx}), we define
\begin{equation}
F^-(X^n_i) = \int_{\xi \in {\R}^-}  \xi {\cal K( \xi)} U^n_i(\xi) d \xi,\quad
F^+(X^n_i) = \int_{\xi \in {\R}^+}  \xi {\cal K(\xi) } U^n_i(\xi) d \xi
\end{equation}
with $U^n_i(\xi)=(U^n_{1,i}(\xi),\ldots,U^n_{N,i}(\xi))^T$.\\

\noindent More precisely the expression of
$F^+(X_i)$ can be written
\begin{align}
F^+(X_i) &= \left(F^+_{k_1^1}(X_i),\ldots,F^+_{k_N^1}(X_i),\right.\nonumber\\ 
& \left .F^+_{k_1^2}(X_i),\ldots,F^+_{k_N^2}(X_i), F^+_{k_1^3}(X_i),\ldots,F^+_{k_N^3}(X_i)\right)^T,
\label{eq:flux}
\end{align}
with
\begin{equation}
F^+_{k_\alpha^{j}}(X_i) = C_{\alpha,i}^{j} l_{\alpha}H_i \int_{w\geq -\frac{u_{\alpha,i}}{c_i}}
(u_{\alpha,i} + w c_{\alpha,i})\chi(w)\ dw\label{eq:FH}.
\end{equation}
This kinetic method is interesting because it gives a very simple and
natural way to propose a numerical flux through the kinetic
interpretation. Indeed, choosing
$$\chi(w)=\frac{1}{2\sqrt{3}} 1_{|w| \leq \sqrt{3}}(w),$$
 the integration in (\ref{eq:FH}) can be done analytically.
 
 However, this method proved to be numerically diffusive. That is why in practice, following \cite{Audusse2010a, Audusse2003}, we rather introduce the upwinding in the biological equations according to the sign of the total mass flux. We introduce then new biological fluxes that we will use instead of (\ref{eq:FH})
 \begin{equation}
F^+_{k_\alpha^{j}}(X_i) = C_{\alpha,i+1/2}^{j} F^{+}_{h_{\alpha}}\label{eq:FH2}
\end{equation}
with $$ F^{+}_{h_{\alpha}} = l_{\alpha}H_i \int_{w\geq -\frac{u_{\alpha,i}}{c_i}}
(u_{\alpha,i} + w c_{\alpha,i})\chi(w)\ dw$$
representing the mass flux in layer $\alpha$ at interface $i+1/2$ (see \cite{Audusse2010a}) and the interfacial quantity being defined in the following manner
$$ C_{\alpha,i+1/2}^{j,n} =
\left\{\begin{array}{ll}
C_{\alpha,i}^{j,n} & \mbox{if } F^{+}_{h_{\alpha}}  \geq 0\\
C_{\alpha,i+1}^{j,n} & \mbox{if } F^{+}_{h_{\alpha}}  < 0.
\end{array}\right.$$

\subsection{The source terms}
 We refer to \cite{Audusse2010a} for the treatment of the mass exchanges terms $S_{e}(X^{n},\tilde{X}^{n+1})$. For the reaction terms no difficulties occur since the biological variables and what they depend on, light for instance, are also projected on the same Galerkin basis. We write
 \begin{equation}
 \int_{z_{\alpha-1/2}}^{z_{\alpha+1/2}}{\int_{x_{i-1/2}}^{x_{i+1/2}}{R(x,z,t^{n})}}=h_{\alpha}\Delta x_{i}R_{i}^{n}
 \end{equation}
 where $R(x,z,t)$ represents one of the three reaction terms we have in the Droop model.
 
 \subsection{Properties}
 Although we are not going to provide any proof in this section we want to recall some important features of the model used in this paper (see \cite{Audusse2010a, Audusse2010b, Sainte-Marie2010} for the detailed proofs). First of all, under a certain CFL condition which means that the  quantity of water leaving the cell during a time step is less than the current water volume of the cell, the water height and the biological concentrations remain non-negative. Second of all, \cite{Audusse2010a} states that a passive tracer satisfies a maximum principle. Finally, we want to specify that a second order scheme in space and time is possible and has been performed in the simulations of section \ref{secsim}. The second order in time is achieved through a classical Heun method. We apply a second order in space by a limited reconstruction of the variables \cite{Audusse2005} based on the prediction of the gradients in each cells, a linear interpolation followed by a limitation procedure.

\section{Simulations}
\label{secsim}

\subsection{Analytical validation on non trivial steady states}
In this section we show that we can find analytical solutions of the coupled problem, provided that we use Euler equations and a simplified model for biology. We want to emphasize the importance of this part. Validating a numerical code is indeed a complex but highly required task when it comes to non trivial situations. It is clear though that analytical solutions of the whole coupled problem cannot be found. However, we propose here a simplified version of a biological model that would be embedded in free surface Euler equations and for which we can find non trivial steady states.

Let (u,w) be the following vector field:
\begin{numcases}
\strut u(x,z) = \alpha\beta\frac{cos(\beta(z-z_{b}))}{sin(\beta H)},\label{uanalytic}\\
w(x,z) = \frac{\alpha\beta}{sin(\beta H)^{2}}\left[ sin(\beta(z-z_{b}))cos(\beta H)\frac{\partial H}{\partial x} + cos(\beta(z-z_{b}))sin(\beta H)\frac{\partial z_{b}}{\partial x}\right]\label{wanalytic},
\end{numcases}
with $\alpha$, $\beta$ being given constants and $H(x)$ being known.
In \cite{JSM_ANALYTIC}, the authors show that it is an analytical solution of the Euler system of equations for free surface flows. In order to validate the coupling with biological equations which is performed in this paper, we propose a simplified stationary coupled model.

Let us consider the following scalar field
$$T(x,z) = e^{-(H-(z-zb))}.$$
We can check by simple calculation that $T$ is solution of
\begin{equation}
\underbrace{\frac{\partial uT}{\partial x} + \frac{\partial wT}{\partial z}}_{\mbox{advection}} = \underbrace{f(x,z)T(x,z)}_{\mbox{reaction}}
\label{eqbioanalytique}
\end{equation}
where
\begin{equation}
f(x,z) = \alpha\beta\frac{cos(\beta(z-z_b))}{sin(\beta H)}\left(\frac{tan(\beta(z-z_b))}{tan(\beta H)}-1\right)\frac{\partial H}{\partial x}.
\end{equation}
Clearly, this equation is a simplified and alternate version of a biological model, with advection and reaction terms. Notice that solutions to (\ref{eqbioanalytique}) represent non trivial equilibria between advection and reaction terms.

Eventually, we compare analytical and numerical results of the following situation: a given hydrodynamic and biological flow is imposed at the left boundary of a 20m long pool with topography $z_{b} = 0.2e^{(x-8)^{2}} - 0.4e^{(x-12)^{2}}$, a given water height is fixed at the right boundary. To solve this analytically we follow \cite{JSM_ANALYTIC}: given the topography  $z_{b}$, between horizontal coordinates $x = 0$ and $x = 20$, we recover $H(x)$ and deduce $u(x,z)$ and $w(x,z)$ for $\alpha = 0.4$ and $\beta = 1.5$ thanks to the above formula (\ref{uanalytic}, \ref{wanalytic}). Numerically, we impose the analytical flux at the left boundary, the analytical height at the left boundary, and run the code for 500s, with 300 nodes and 20 layers (thus 6000 vertices), while initial conditions were set up to zero for every variable ($u$, $w$, $T$). We see in Fig.\ref{analyticSol} that we recover numerically the hydrological and tracer steady states solutions for this simplified coupled model. Therefore, we consider that our method and our code are likely to produce valid results.

\begin{figure}[h!]
\begin{center}
\includegraphics[scale = 0.8]{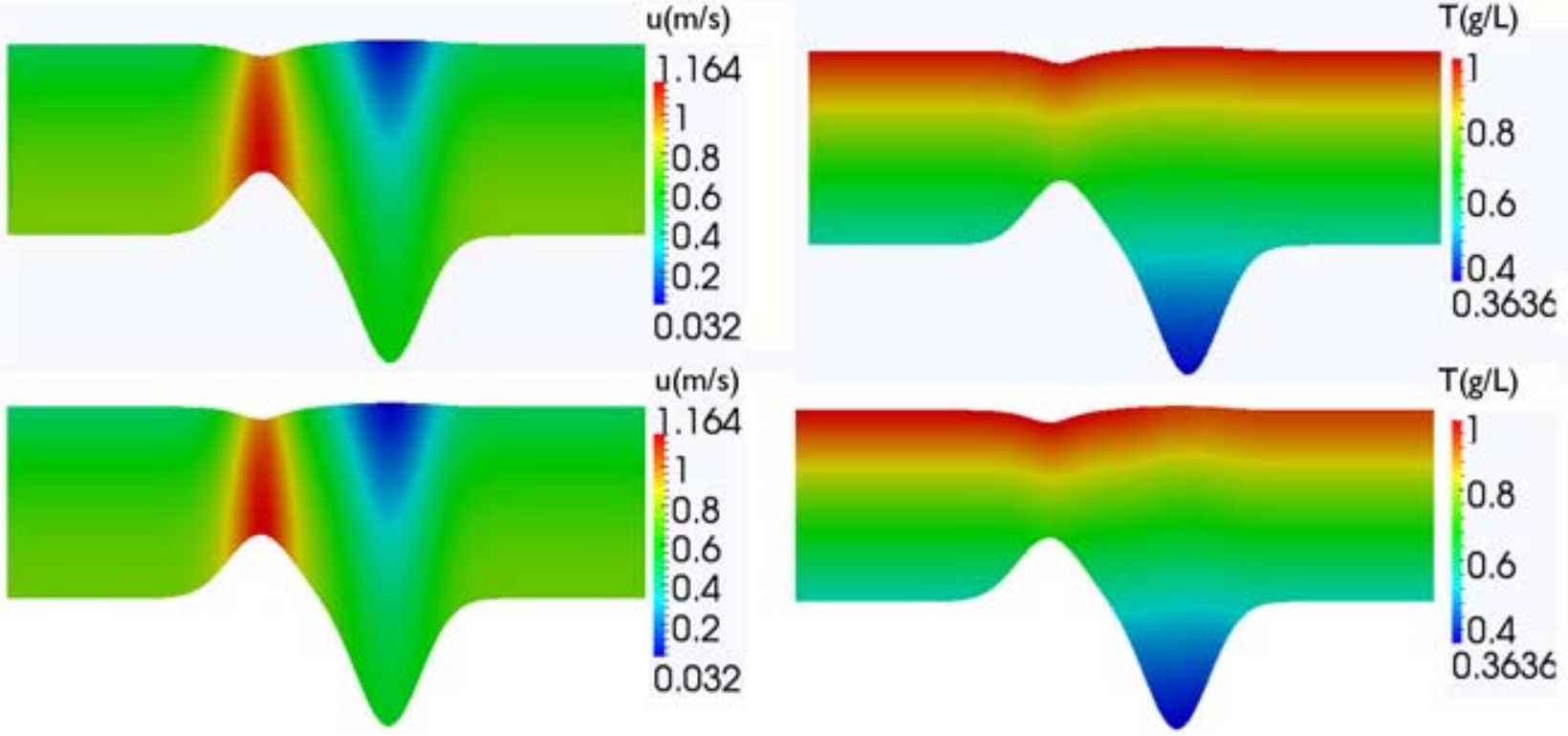}
\caption{Analytical (top) and numerical(bottom) steady states solutions for hydrodynamics and tracer in case $\alpha = 0.4$, $ \beta = 1.5$ and $z_{b} = 0.2e^{(x-8)^{2}} - 0.4e^{(x-12)^{2}}$. A given flow is imposed at the left boundary whereas the output boundary condition concerns the water height.}
\label{analyticSol}
\end{center}
\end{figure}

\subsection{Numerical simulations of a raceway}
\subsubsection{Light}
As explained in section \ref{secmodel}, the light intensity at a particular point will depend both on the amount of water above this point (i.e. related to the depth) and on the quantity of encountered chlorophyll (proportional to the concentration in intracellular nitrogen). Additionally, we explain in section \ref{secnumscheme} that a Galerkin approximation of the variables is carried out. Hence light is also discretized along the depth by layers. We show in Fig.\ref{light} the light profile for different values of $\gamma(I^{*})$. Recall that this value is related to the average irradiance perceived by microalgae the day before (photoadaptation phenomenon). The parameters used are exposed in Table \ref{tableLight}. We see for any curve the the exponential decay. Moreover, this figure shows that the more exposed to light the  microalgae were, the more receptive to it they are (no photoinhibition occurs in those ranges of irradiance).

\begin{minipage}{.45\textwidth}\centering
\includegraphics[height = 5.5cm]{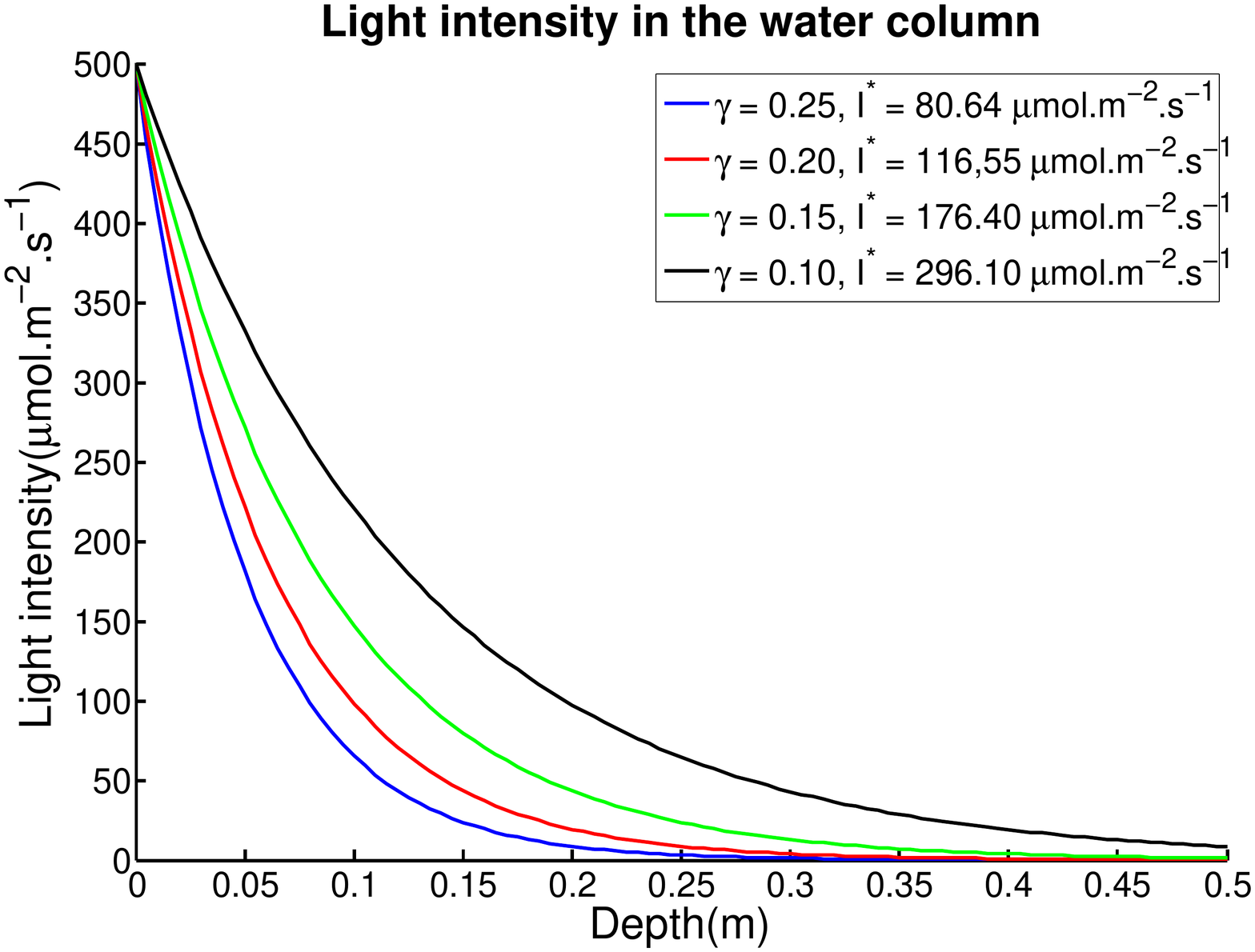}
\captionof{figure}{Light intensity in the water column for a concentration in nitrogen equal to 5.0 $gN.m^{-3}$. The figure shows the exponential decay of light with depth for different values of $\gamma(I^{*}).$}
\label{light}
\end{minipage}
\begin{minipage}{.45\textwidth}\centering
     \begin{tabular}{|c|c|c|}
	\hline
	Parameter & Value & Unit \\
	\hline
	$I_{0,max}$ & 500 & $\mu mol.m^{-2}.s^{-1}$ \\
	a & 16.2 & $m^{2}.gChl^{-1}$ \\
	b & 0.087 & $m^{-1}$ \\ 
	$C^{2}$ & 5.0 & $gN.m^{-3}$   \\   
	\hline
      \end{tabular}
\captionof{table}{Parameters used for the computation of light.}
\label{tableLight}
\end{minipage}

\subsubsection{Paddlewheel}
\label{paddle}
We aim at representing the kind of raceway shown in Fig.\ref{raceway3D}.
 \begin{figure}[h!]
\begin{center}
\includegraphics[scale = 0.3]{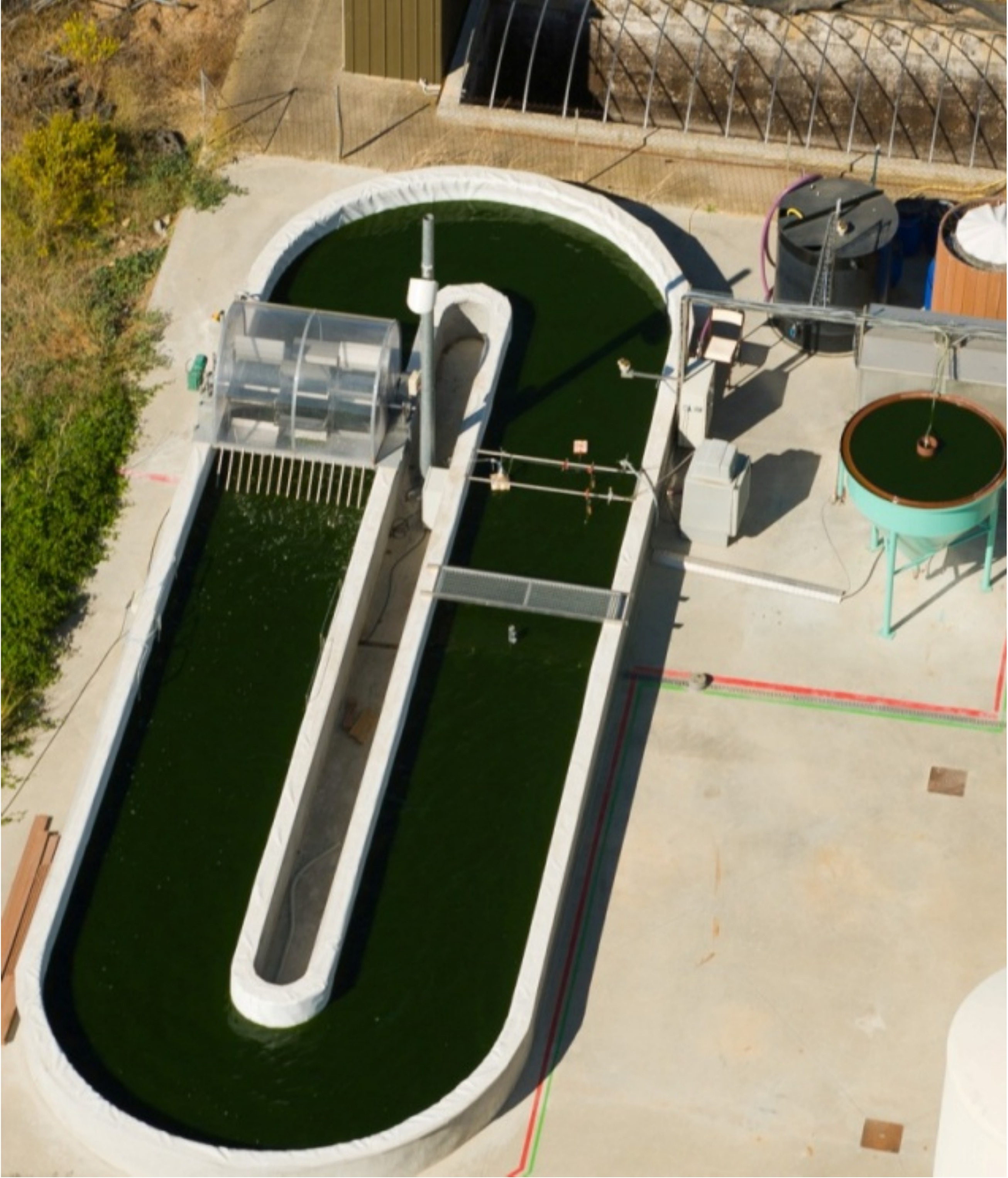} 
\caption{A typical raceway for cultivating microalgae. Notice the paddlewheel which mixes the culture suspension. Picture from INRA (ANR Symbiose project).}
\label{raceway3D}
\end{center}
\end{figure}
In order to model it in 2D, periodic conditions are applied on a rectangular pool containing a paddlewheel in its first half (Fig.\ref{forceoutlook}). The wheel is not modelled physically but is represented by a force that is able to mimic its effect and give the system an equivalent energy. Therefore the following expression is assumed:
\begin{numcases}
\strut 
F_x (x,z,t) = F \left(\sqrt{(x-x_{wheel})^2+(z-z_{wheel})^2} \omega\right)^{2}\cos\theta\\
F_z (x,z,t) = F \left(\sqrt{(x-x_{wheel})^2+(z-z_{wheel})^2} \omega \right)^{2}\sin\theta
\end{numcases}
where $F$ is a constant, $\theta$ is the angle between the blade and the vertical direction, $\omega = \dot{\theta}$. The force is normal to the blade and depends on the square of the velocity of the point on the blade: the further it is from the center of the wheel, the bigger is the force. Therefore, the energy provided by the wheel is be proportional to the cube of the velocity. To include it in the model, the process is similar to what is explained in section \ref{numscheme}: the force is added in the Navier-Stokes equations in order to derive again the multilayer model. This adds a source term in the x-momentum equation and change the expression of the pressure given by the z-momentum equation. We remind the reader that non-hydrostatic terms have not been taken into account in the derivation of the model (though it is described in \cite{Sainte-Marie2010}). The obtained results in Fig.\ref{velresults} let us think that this may  be appropriate to model the paddlewheel effect.

 \begin{figure}[h!]
\begin{center}
\begin{tabular}{cc}
\includegraphics[scale = 0.40]{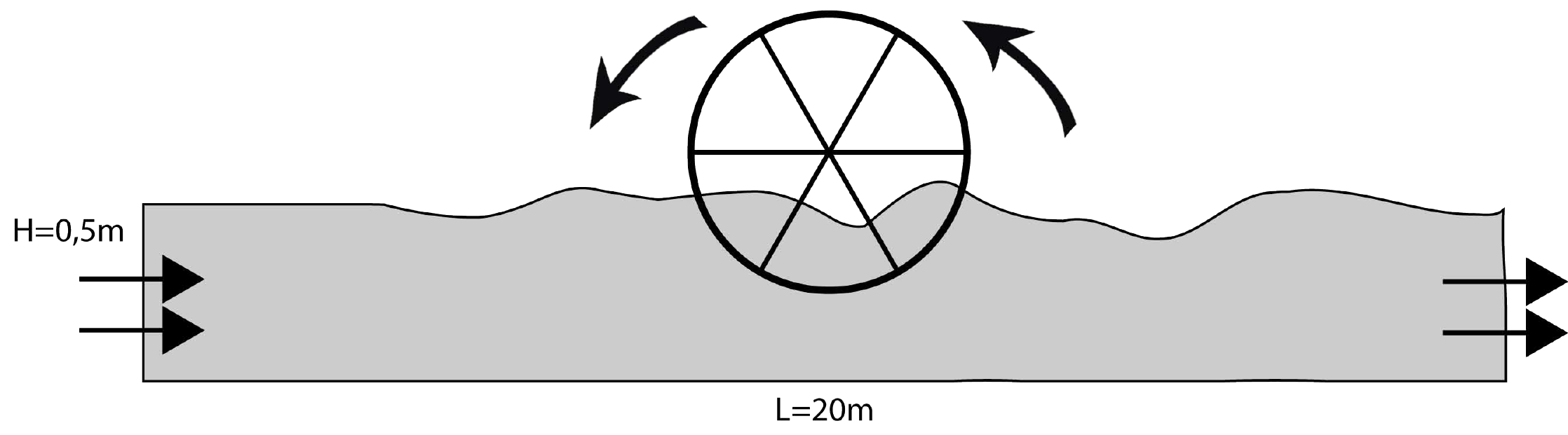}  & 
\includegraphics[scale = 0.25]{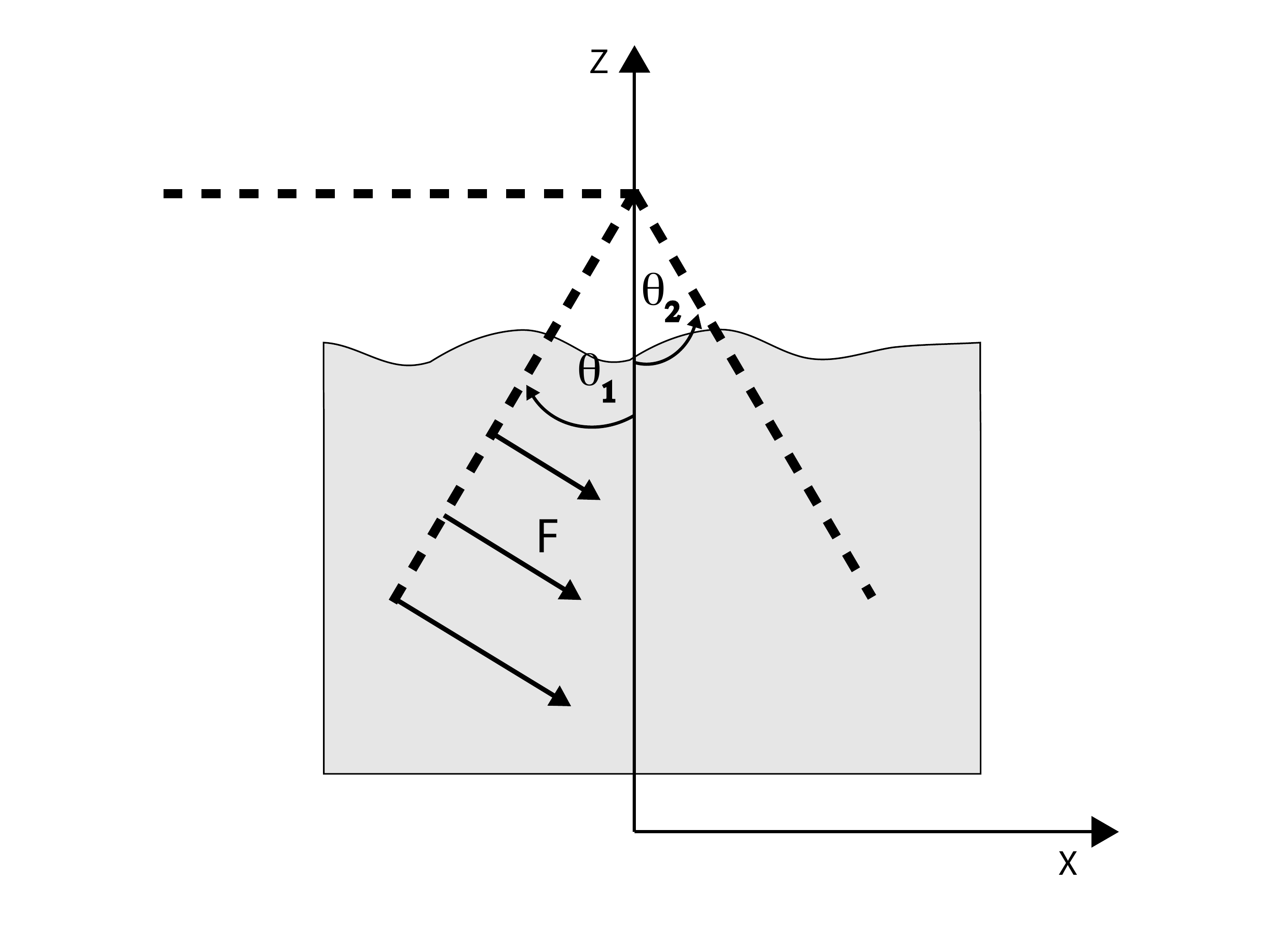}
\end{tabular}
\caption{Left: Pool and wheel dimensions and positions. Right: Outlook of the force applied to model the effect of the paddlewheel, supposed to be located in the first half of the pool, with maximum efficiency at the end of the blade.}
\label{forceoutlook}
\end{center}
\end{figure}

We performed several simulations for different paddlewheel angular velocities. We use a pool with length 20m and height 0,5m. We show in Fig \ref{velresults} that we are able to capture realistic hydrodynamics: a laminar flow of reasonable horizontal speed far from the wheel and a turbulent flow close to it. Concerning the hydrodynamics parameters, we take into account horizontal and vertical viscosity ($\mu =0.001 m^{2}.s^{-1}$)and a Navier-type bottom friction( $\kappa = 0.01m.s^{-1}$). Besides, $\omega = 0.85$rad/s. 
 \begin{figure}[h!]
\begin{center}
\includegraphics[scale = 0.50]{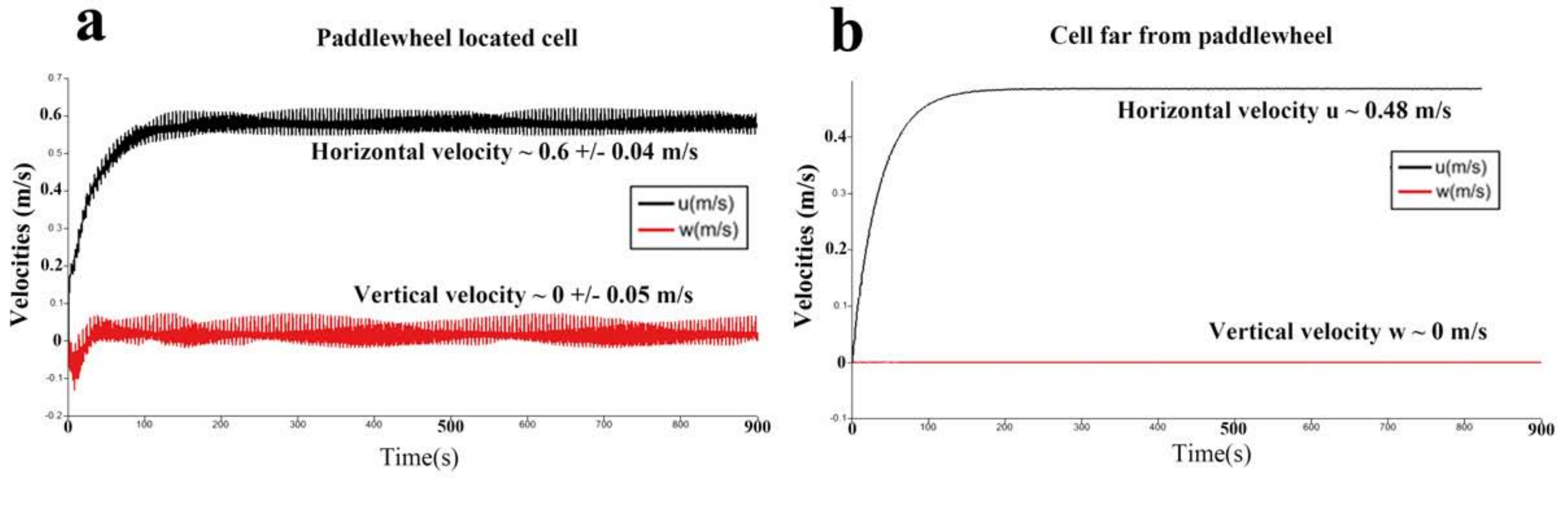}
\caption{(a): Velocities along vertical and horizontal axis in a cell located near from the wheel rotating at angular speed $\omega = 0.8 rad/s$. The flow is very turbulent. (b): Velocities along vertical and horizontal axis in a cell located far from the wheel. An asymptotic value of $0.48m.s^{-1}$ is reached.}
\label{velresults}
\end{center}
\end{figure}

In order to have a first idea of what is happening in the fluid, we add a tracer which is advected and slightly diffused. Fig \ref{film} illustrates the effect of the wheel on the mixing. In particular,  after several minutes, the pool seems well homogenized (here again $\omega  =0.85 rad/s$).

\begin{figure}
\begin{center}
\includegraphics[width = \textwidth]{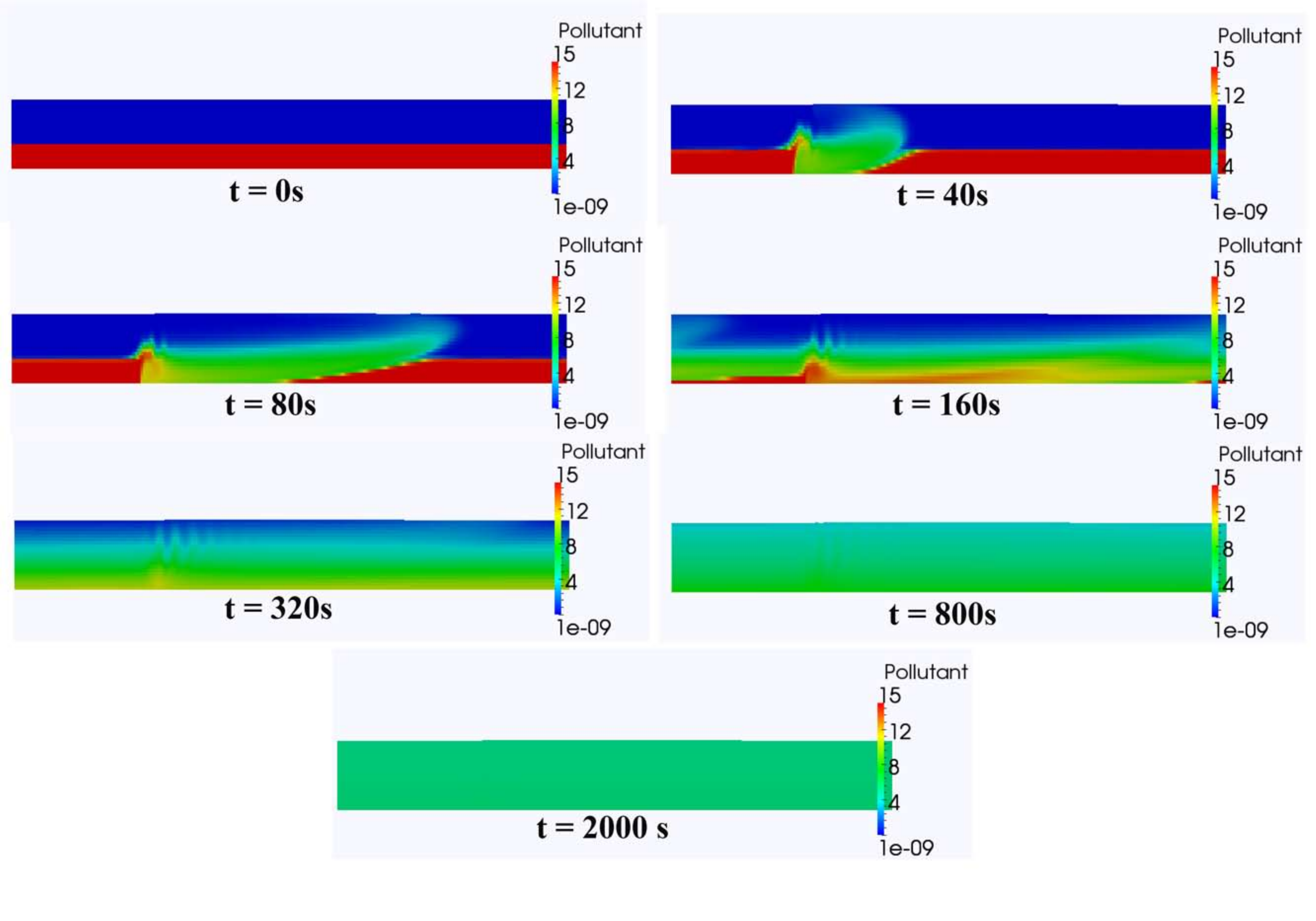}
\caption{Snapshots of the tracer (pollutant) concentration in a raceway set into motion by a paddlwheel of angular velocity $\omega = 0.85 rad/s$. It is clear that after several minutes, the raceway is totally homogeneous. Therefore, the paddlewheel has indeed the required effect on the mixing.}
\label{film}
\end{center}
\end{figure}

\subsection{Results}

Eventually, we performed numerical simulations of the whole coupled model explained in section \ref{secmodel} and discretized in section \ref{secnumscheme}. We use the same pool with length 20m and height 0,5m. The height is chosen such that at the bottom of the pond, the respiration rate is close to the growth rate. We perform several 20 days simulations for different agitations, different initial conditions and we compare them to reference simulations i.e. without paddlewheel. The parameters used for the simulation (concerning the biological system) are exposed in Table \ref{simuparam}. The relevant details of every simulation are in Table.\ref{simucond}. The results are depicted in Fig.\ref{simresults}.

The initial concentrations of particulate and dissolved nitrogen are the same for the six simulations. The initial microalgal carbon concentration varies (the initial internal nitrogen quota is therefore changing as well), and different agitation velocities are tested. The plots represent the average concentrations in the raceway (see local concentration comparison between upper and bottom layer in Fig.\ref{local}). The carbon curves (Fig.\ref{simresults}-a) show in every situation that agitation leads to better productivity. This is explained by the lack of nutrients in the upper layers at some point. However, regarding the initial internal quota (Fig.\ref{simresults}-b), the time when agitation actually enhances productivity varies. This phenomena is due to the fact that  the internal nutrient pool, if not filled enough(low quota $q$), will tend to increase further by absorbing more external nitrogen. Thus leading two main consequences: the extracellular nutrient concentration diminishes and the intracellular nitrogen increases. Since chlorophyll is positively correlated with the latter biological variable, the light can not penetrate so deep anymore. We point out that the model is not able to differentiate between several agitation velocities. Ideas to explain and overcome this limitation are suggested in section \ref{seclag}.

From these series of simulations we can conclude that our model is capable of reproducing coherent results, in the hydrodynamical part (adequate asymptotic velocity, turbulences near the wheel and laminar flow far from it) as well as in the biological concentrations. Nevertheless, we can only consider those results as preliminary, since no quantitative comparison with any data has been provided. Future work would include adaptation of the biological model given the hydrodynamical results obtained in this study. Moreover, other variables should be added to the system. For instance temperature, which could have a non negligible effect, sedimentation etc. Finally, experimental data should help us calibrate more accurately the raceway parameters and extend it to three dimensions.
\begin{table}[h!]
    \begin{center}		
      \begin{tabular}{|c|c|c|}
	\hline
	Parameter & Value & Unit \\
	\hline
	$\tilde{\mu}$ & 1.7 & $day^{-1}$ \\
	$Q_0$ & 0.050 & $gN.gC^{-1}$ \\
	$Q_l$ & 0.25 & $gN.gC^{-1}$ \\
	$K_{iI}$ & 295 & $\mu mol.m^{-2}.s^{-1}$ \\
	$K_{sI}$ & 70 & $\mu mol.m^{-2}.s^{-1}$ \\
	$\overline{\lambda}$ & 0.073 & $gN.gC^{-1}.day^{-1}$ \\
	$K_s$ & 0.0012 & $gN.m^{-3}$ \\
	R & 0.0081 & $day^{-1}$ \\
	$I_{0,max}$ & 500 & $\mu mol.m^{-2}.s^{-1}$ \\
	$\gamma(I^*)$ & 0.25 & $gChl.gN^{-1}$ \\
	a & 16.2 & $m^{2}.gChl^{-1}$ \\
	b & 0.087 & $m^{-1}$ \\         
	\hline
      \end{tabular}
    \end{center}
  \caption{Parameters for the simulations.}
  \label{simuparam}
\end{table}

\begin{table}
\begin{center}
\def\arraystretch{1.5}
\begin{tabular}{c|cccccc}

   & Simu 1 & Simu 2 & Simu 3 & Simu 4 &Simu 5 & Simu 6  \\
   \hline
 $C^{1}_{0} (g.m^{-3})$  & 25 & 25 & 50 & 50 & 83 & 83  \\
 $\left(\frac{C^{2}}{C^{1}}\right)_{0}(gN.gC^{-1})$  & 0.2 & 0.2 & 0.1 & 0.1 & 0.06 & 0.06 \\
 $C^{3}_{0}(g.m^{-3})$ & 5 & 5 & 5 & 5 & 5 & 5  \\
 Agitation & no & yes & no & yes & no & yes \\
 $C^{1}_{f} (g.m^{-3})$ & 60 & 74 & 79 & 103 & 100 & 129 \\
 $\left(\frac{C^{2}}{C^{1}}\right)_{f}(gN.gC^{-1})$ & 0.115 & 0.120 & 0.110 & 0.085 & 0.087 & 0.065 \\
 $C^{3}_{f}(g.m^{-3})$ & 3.77 & 0.0 & 0.0 & 0.0 & 0.0 & 0.0 
\end{tabular}
\caption{Initial conditions and final concentrations for the set of 6 simulations we performed. The differences lie on the one hand on the fact that the water is agitated or not and on the other hand on the initial carbon concentration, which changes the internal quota (nitrogen concentration is always equal to $5 ((g.m^{-3})$)}
\label{simucond}
\end{center}
\end{table}

 \begin{figure}[h!]
\begin{center}
\includegraphics[scale = 0.30]{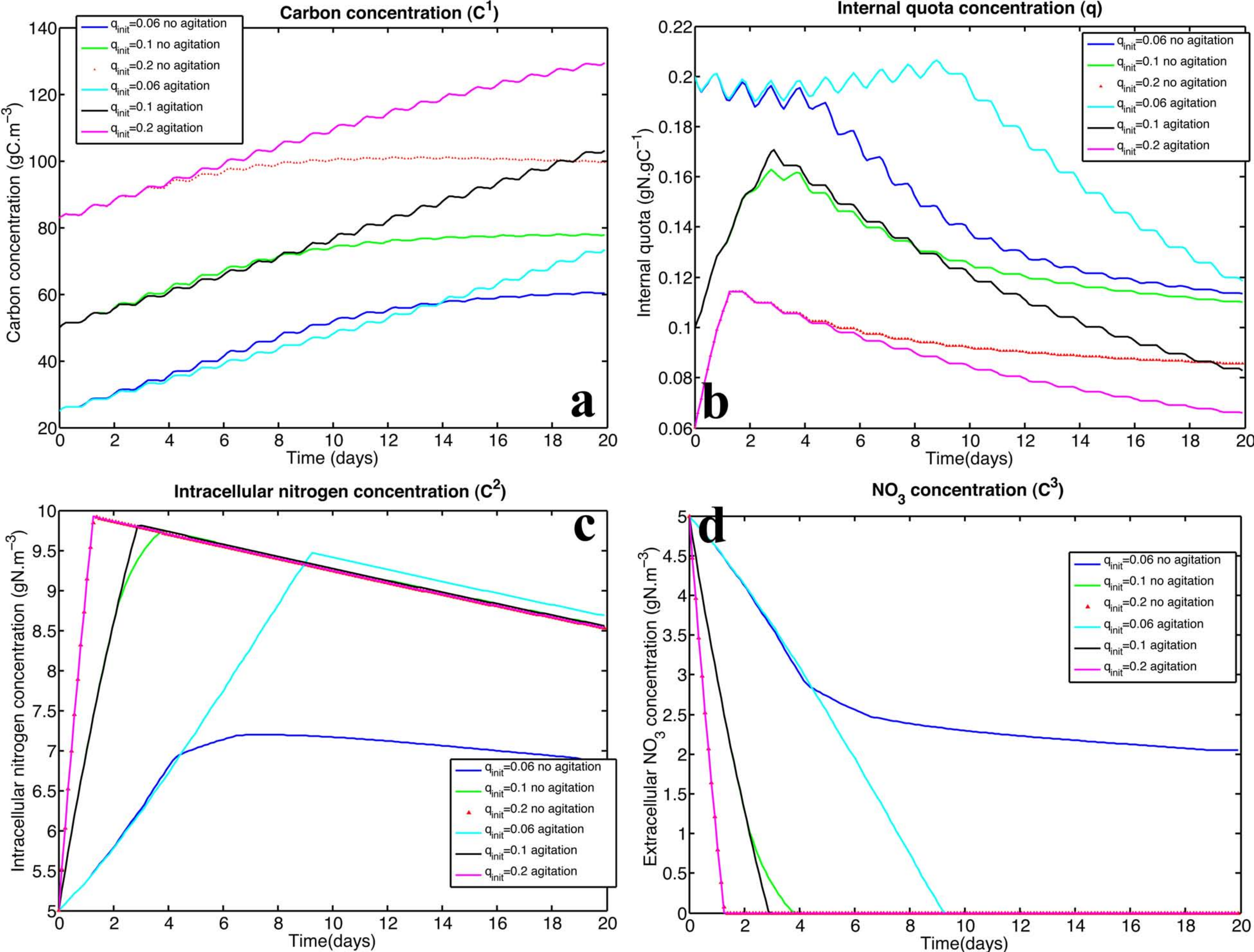}  
\caption{(a): Carbon concentration; (b): Internal quota $q$; (c): nitrogen concentration; (d): substrate concentration($NO_{3}$). Those plots illustrate the average concentrations in the raceway for 6 simulations. Three were carried out without agitation, and the other three had the agitation term. In each situation, agitation, leading to homogenization leads to a better productivity. However, for certain initial conditions, the improvement is quite slow (after several days), since the biological variables do not evolve as quickly as hydrodynamics does.}
\label{simresults}
\end{center}
\end{figure}

 \begin{figure}[h!]
\begin{center}
\includegraphics[scale = 0.3]{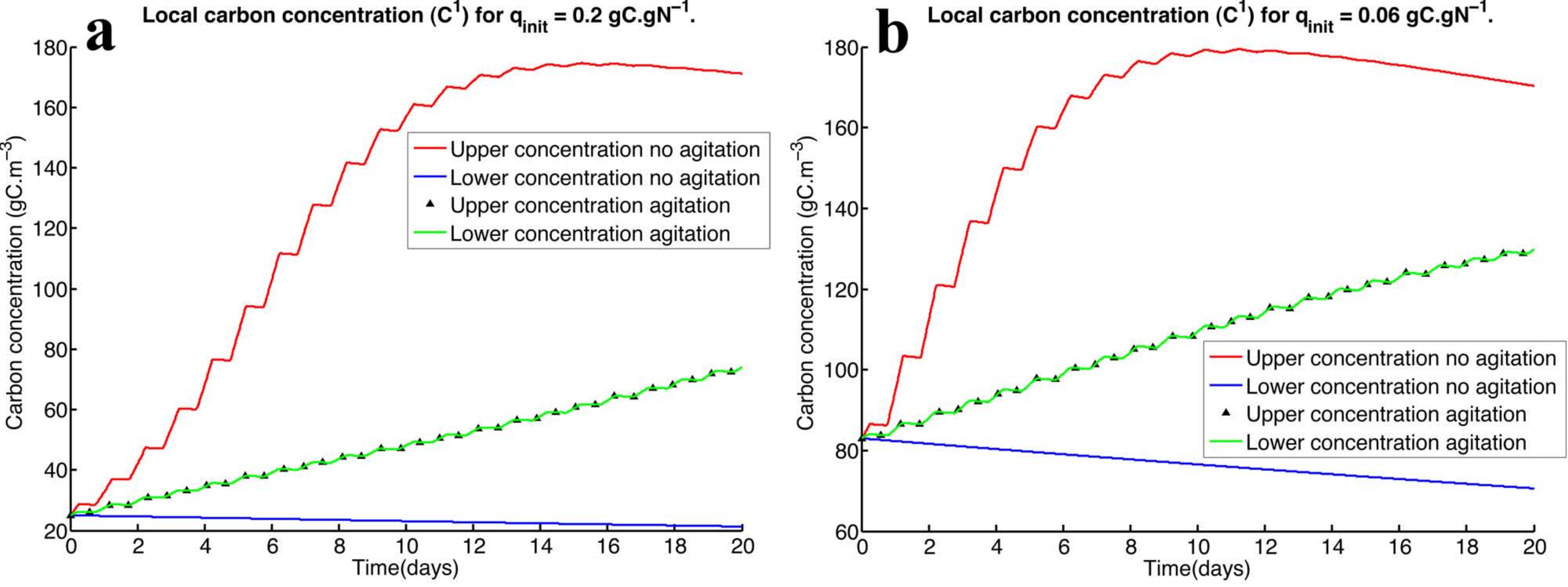}  
\caption{(a): Local carbon concentration when $q_{init} = 0.2 gN.gC^{-1}$, in the bottom and upper layer, with and without agitation. (b): Local carbon concentration when $q_{init} = 0.06 gN.gC^{-1}$, in the bottom and upper layer, with and without agitation. We see clearly the homogenization due to agitation. In both cases, we also see that the bottom layer when not agitated does not vary so much, which means that we are close to the point when respiration compensates growth.}
\label{local}
\end{center}
\end{figure}

\section{Lagrangian approach}
\label{seclag}

In order to better understand the physiology of the microalgae submitted to variable light and nutrient stresses, and improve the biological models in these realistic conditions, it is crucial to study these organisms at individual scale when excited by such signals whose frequency is much faster than in natural environments such as lakes or oceans. However, it is very tricky to measure the light and nutrient signal perceived by a cell in a raceway. A Lagrangian approach derived from the previous model can lead to the reconstruction of cell trajectories. From this information, it is then possible to derive the signals that microalgal cells undergo.

In order to follow the position of a particle, we simply need to integrate the following equation. If $M(t)$ is the position of particle $M$ at time t, then we have:
\begin{equation}
		\frac{dM(t)}{dt} = v(M(t),t)
\end{equation}
where $v(x,t)$ is the eulerian velocity field at position $x$, time $t$. We do not add to those particles any Brownian motion that would refer to the diffusion of biological concentrations at the macroscopic scale. Therefore, more realistic Lagrangian trajectories should look more noisy than what we get but the general behaviour is well represented.

The global simulation in the previous section could not differentiate between several paddlewheel velocities. It is likely that the biological model is not affected, provided that the mixing is efficient. Therefore we try to understand what could be those differences. For several angular velocities, we perform a one hour simulation and build the trajectories afterwards. One hundred particles are equally distributed in the pool along the depth dimension. Fig.\ref{bar} shows the distribution of particles against the percentage of time spent under more than 50\% of the incident light (we will call it high enlightenment). We notice that the wheel velocity has an influence on the number of particles which never undergo high enlightenment (38 particles over 100 for $\omega = 0.5$ and only 13 particles over 100 for $\omega = 1.0$). But globally, the particles are distributed around 20\% of their time under high enlightenment. This is what we can expect from a good mixing since the light is exponentially decreasing and getting 50\% of the incident light means being in the first 10cm from the surface of the pool (over 50cm).

 \begin{figure}[h!]
\begin{center}
\includegraphics[scale = 0.6]{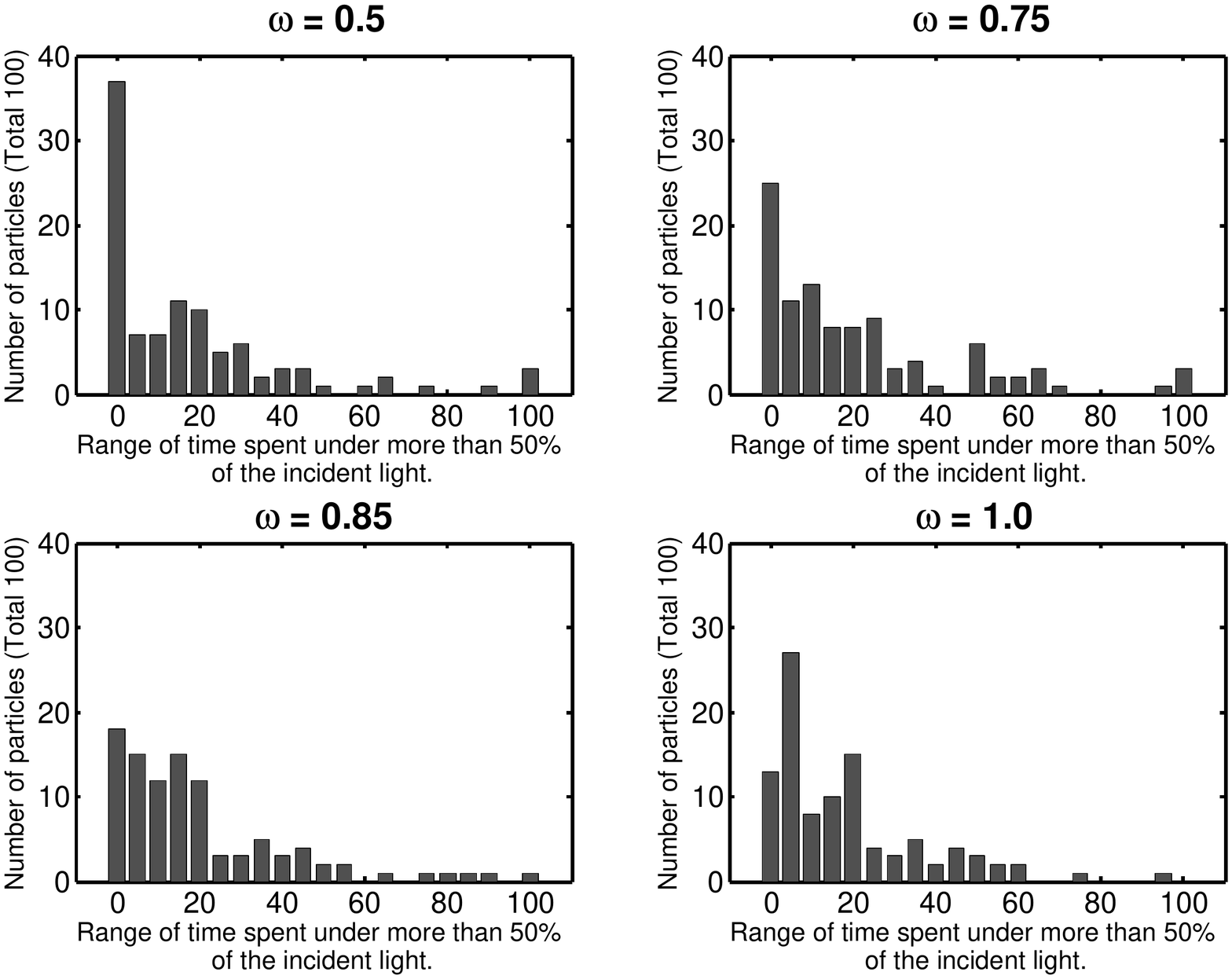}  
\caption{Number of particles for each class of enlightenment for four wheel angular velocities(0.5, 0.75, 0.85 and 1.0 $rad.s^{-1}$). One class represents a range of 5\%. Being for instance in the class 20\%-25\% means that the particle spent between 20 and 25\% of her time under hight enlightenment (more than 50\% of the incident light).}
\label{bar}
\end{center}
\end{figure}

In Fig.\ref{omegacurves} we plot different indicators for eight velocities, established during a one hour simulation. First of all, Fig.\ref{omegacurves}.a illustrates the average velocity at the end of the simulation. It turns out to lie in realistic ranges (0.2 $m.s^{-1}$ to 0.7 $m.s^{-1}$). Second of all Fig.\ref{omegacurves}.b represents the average proportion of time spent for any particle under high enlightenment. From this curve we deduce that the global percentage of high enlightenment may not vary much between $\omega = 0.5$ and $\omega = 1$. The mixing is indeed well carried on (as shown in previous section) and in average, the particles have the same history. Finally, Fig.\ref{omegacurves}.c depicts the number of times the particles switch from low enlightenment to high enlightenment during one hour. This last indicator gives insights about the duration of the high enlightenment periods. The greater it is, the shorter but more numerous were the high light instants (the particles switches more often).
 \begin{figure}[h!]
\begin{center}
\includegraphics[scale = 0.3]{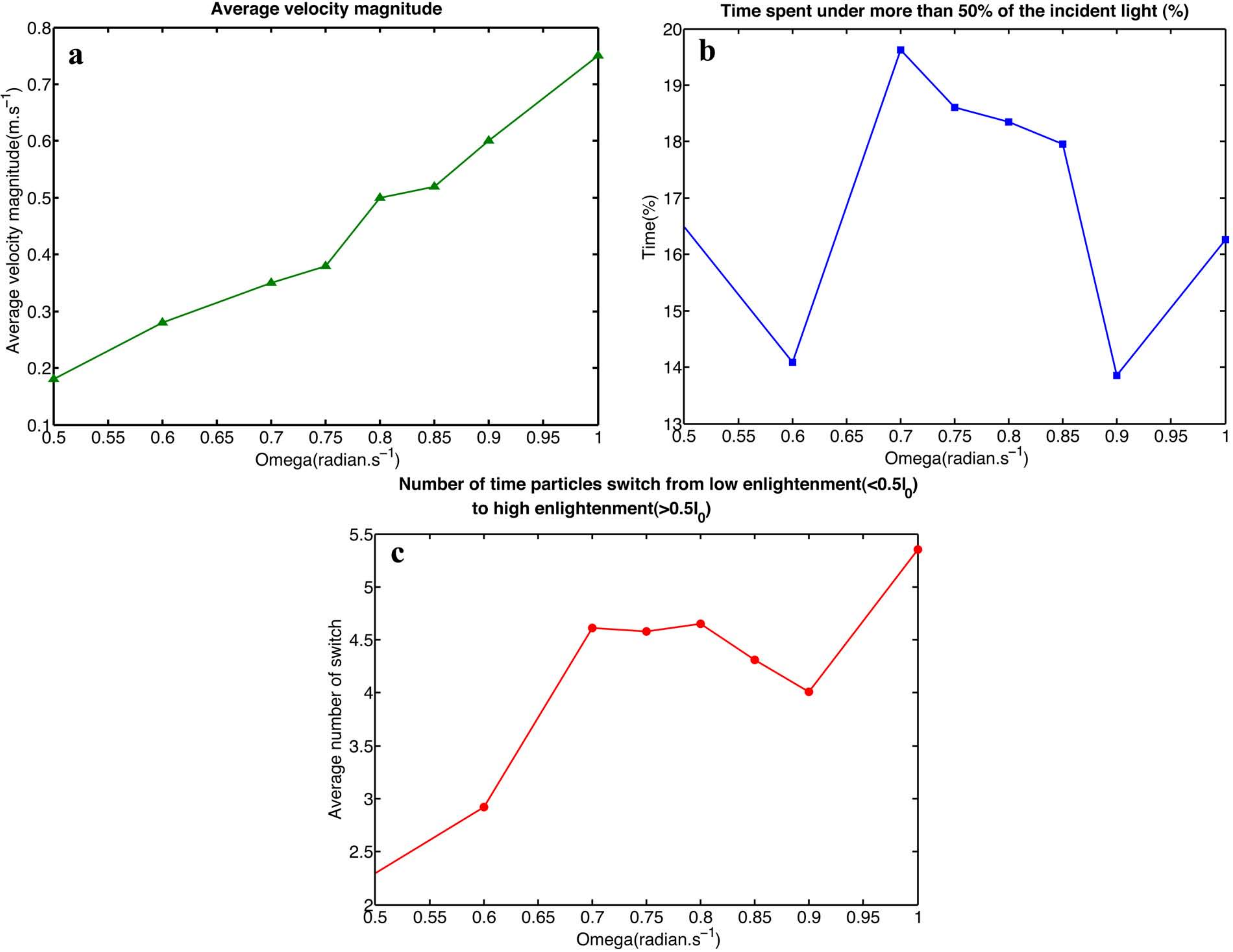}  
\caption{(a): Average velocity after one hour for 8 velocities. (b): Percentage of time spent under more than 50\% of $I_{0}$. (c): Number of times a particles switches from a situation where it perceived less than 50\% of light intensity to a situation where it perceives more during one hour.}
\label{omegacurves}
\end{center}
\end{figure}

Finally, the trajectories of three particles are depicted in Fig.\ref{tracking}.a. From those trajectories we can extract two important informations. First of all, between two passages around the paddlewheel, the particle seems to stay at constant depth, thus enforcing the fact that the flow is laminar apart from the wheel. Second of all, we clearly see that the depth of the particle is suddenly modified by the wheel, giving rise to abrupt changes in the enlightenment. Fig.\ref{tracking}.b shows the light received by those three particles in the case where the average intracellular nitrogen concentration is 5 $gN.m^{-3}$ and $\gamma(I^{*}) = 0.1$ $gChl.gN^{-1}$ (high irradiance the day before the simulation).
 \begin{figure}[h!]
\begin{center}
\includegraphics[scale = 0.3]{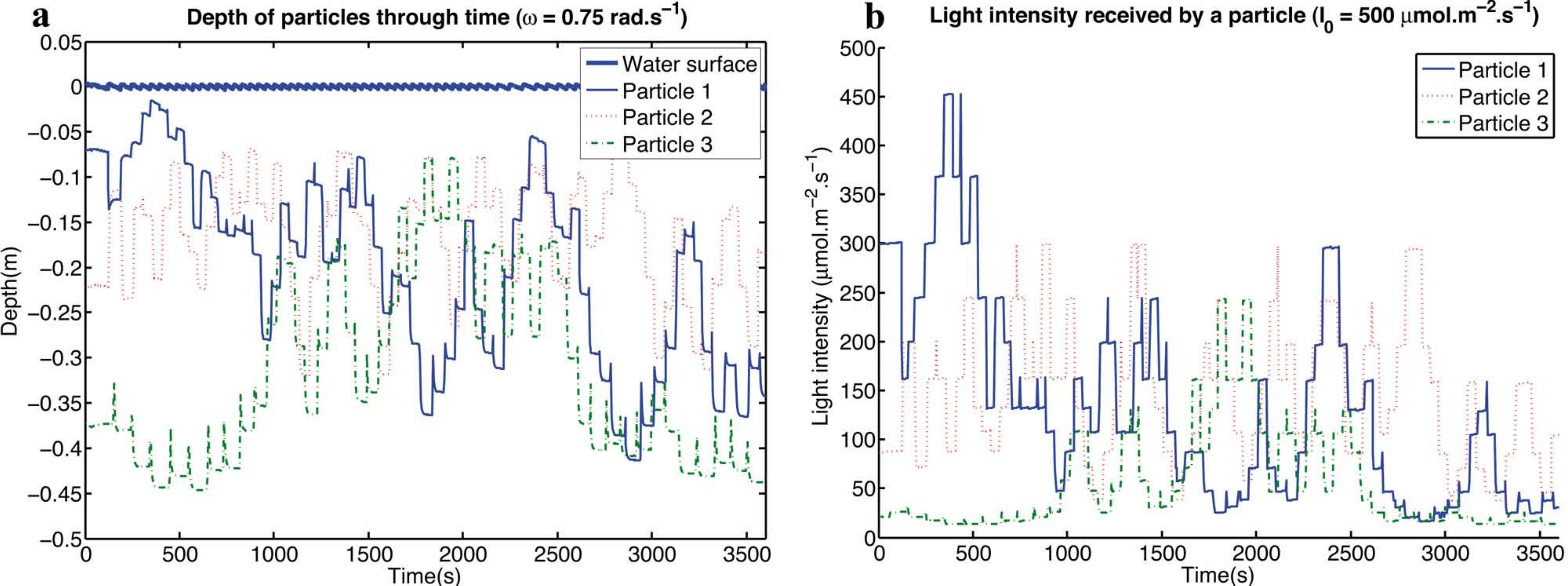}  
\caption{(a):Trajectories of three particles during the simulations. The large curve represents the water surface at the middle of the pool. The other plot is the height of a given particle through time. The algae undergo sudden changes of depth every time it meets the wheel.(b): Perceived light from the microalgae. Particles are subject to even greater irradiance changes since the light is exponentially decaying.}
\label{tracking}
\end{center}
\end{figure}

Fig.\ref{light} shows that light intensity is very low in the 20 deepest centimeters. Regarding our results on Fig.\ref{tracking}, we assume that microalgae should spend one lap at high light and then several laps at low light. Since the asymptotic velocity is in the range $0.3m.s^{-1} - 0.8m.s^{-1}$, a whole lap is done in the order of a minute. Therefore, algae are faced to light changes with a time scale in the range of ten minutes. As regards our modeling problem, this tells us that a biological model valid for these fast time scale could lead to more precise results. Additional experiments forcing microalgae with typical light signal deduced from Figure \ref{light} must therefore been carried out and support a microalgae modelling at fast time scale \cite{Esposito2009}.

\section{Conclusion}

In this paper we derived a new model coupling hydrodynamics to biology in two dimensions. It provides new insights to better understand and represent this nonlinear and non stationary complex process. The multilayer model adequately represents the high vertical heterogeneity characterizing the agitated raceway. A key point of our approach, is that we were able to provide analytical solutions which validated its numerical integration. The results show that water agitating through the paddlewheel has an effect on the growth of algae, particularly because of the lack of nutrients at the surface versus the lack of light around the pool's floor. One of the outcome of this work is the identification of realistic light signals to which microalgae are faced. Lab scale experiments will be performed to assess the impact of such high frequency light signals on microalgae. It is worth remarking that similar works have been carried out for photobioreactor \cite{Perner2002b,Perner2007a}, leading to the identification of much faster time scales in the range of the second. The comparison with experimental data would  be of great relevance to better calibrate the hydrodynamics, and later on improve the biological predictions of the model.  It is clear however that many parameters need to be taken into account in order to increase the model prediction capacity, for instance temperature. It was here considered as a passive tracer, only advected and diffused, with no particular effect on the biology (temperature does not appear in the growth or respiration rate). The sunlight effect on water temperature will be taken into account in a next stage since it deeply affects microalgae growth. Moreover, some microalgae species do not swim in the water and tend to sediment. This property could increase the beneficial effect of the wheel compared to a situation where the wheel is very slow or absent. Computational time is another issue which has to be improved in order to use the model {\it e.g.} for process optimal design. Up to now, schemes we are developing are only explicit. We are therefore constrained with a restrictive CFL condition. Improvement towards implicit schemes is also a future concern. 

\section{acknowledgement}
The authors would like to thank the support of INRIA ARC Nautilus (\url{http://www-roc.inria.fr/bang/Nautilus/?page=accueil}) together with the ANR Symbiose project (\url{http://www.anr-symbiose.org}).

\clearpage
\bibliographystyle{plain}
\bibliography{bibhydrobio2D}

\end{document}